\documentclass[a4paper, 11pt]{amsart}
\usepackage{amsmath,amsthm,amssymb}
\usepackage{graphicx,color}
\usepackage[all]{xy}
\usepackage{hyperref}
\usepackage[margin=3cm]{geometry}
\usepackage{tikz,enumerate}

\theoremstyle{plain}
\newtheorem{theorem}{Theorem}[section]
\newtheorem{proposition}[theorem]{Proposition}
\newtheorem{lemma}[theorem]{Lemma}
\newtheorem{corollary}[theorem]{Corollary}
\numberwithin{equation}{section}
\newtheorem*{maintheorem*}{Main Theorem}

\theoremstyle{definition}
\newtheorem{definition}[theorem]{Definition}
\newtheorem{remark}[theorem]{Remark}
\newtheorem{example}[theorem]{Example}

\newcommand{\C}{\mathbb{C}}
\newcommand{\Q}{\mathbb{Q}}
\newcommand{\R}{\mathbb{R}}
\newcommand{\Z}{\mathbb{Z}}
\newcommand{\e}{\mathbf{e}}
\newcommand{\fF}{\mathfrak{F}}
\newcommand{\RP}{\mathbb{R}\mathbb{P}}

\def\wt{\widetilde}
\def\ul{\underline}
\def\ol{\overline}
\def\RZ{\mathbb{R}\mathcal{Z}}

\DeclareMathOperator{\card}{card}

\DeclareMathOperator{\row}{row}
\DeclareMathOperator{\Hom}{Hom}
\DeclareMathOperator{\Cri}{Cri}
\DeclareMathOperator{\PL}{PL}
\DeclareMathOperator{\Link}{Link}

\begin{document}
   \title[Cohomology of real toric manifolds and small covers]{Integral cohomology groups of real toric manifolds and small covers}
\author[L. Cai]{Li Cai}
\address{Department of Mathematical Sciences, Xi'an Jiaotong-Liverpool University, Suzhou 215123, Jiangsu, China}
\email{Li.Cai@xjtlu.edu.cn}

\author[S. Choi]{Suyoung Choi}
\address{Department of Mathematics, Ajou University, 206 Worldcup-ro, Suwon 16499, South Korea}
\email{schoi@ajou.ac.kr}

\thanks{The first named author was supported by the National Research Foundation of Korea Grant funded by the Korean Government (NRF-2019R1A2C2010989).}



\newcommand\numberthis{\addtocounter{equation}{1}\tag{\theequation}}

\maketitle
  \begin{abstract}
    For a simplicial complex $K$ with $m$ vertices, there is a canonical $\Z_2^m$-space known as a real moment angle complex $\RZ_K$.
    In this paper, we consider the quotient spaces $Y=\RZ_K / \Z_2^{k}$, where $K$ is a pure shellable complex and $\Z_2^k \subset \Z_2^m$ is a maximal free action on $\RZ_K$.
    A typical example of such spaces is a small cover, where a small cover is known as a topological analog of a real toric manifold.
    We compute the integral cohomology group of $Y$ by using the PL cell decomposition obtained from a shelling of $K$.
    In addition, we compute the Bockstein spectral sequence of $Y$ explicitly.
  \end{abstract}

\tableofcontents


\section{Introduction}
One of the most important classes in toric geometry is a class of projective non-singular complete toric varieties.
A \emph{toric variety} is an algebraic normal variety that admits an action of $(\C^\ast)^n$ with an open dense orbit, where $\C^\ast = \C \setminus \{0\}$.
A non-singular complete toric variety is simply called a \emph{toric manifold}.
The fundamental theorem says that a projective toric manifold $X$ of complex dimension $n$ can be represented by an $n$-dimensional simple convex polytope $P_X$ in the Euclidean space $\R^n$ such that, at each vertex, the outgoing normal vectors of facets containing the vertex become part of a $\Z^n$-basis.

Another important object in toric geometry is the real locus of a toric manifold $X$, denoted by $X^\R$. We simply call $X^\R$ a \emph{real toric manifold}.
It is known that $X^\R$ is a real variety and a smooth manifold.
If $X$ is a projective toric manifold associated to $P_X$   of dimension $n$, then $X^\R$ is completely determined as a variety from partial information of $P_X$.
More precisely, if $P_X$ has $m$ facets $F_1, F_2, \ldots, F_m$, then $X^\R$ is determined by a pair of simplicial complex $K_X$ and a linear map $\lambda_X \colon \Z_2^m \to \Z_2^n$  such that:
\begin{itemize}
  \item $K_X$ is the boundary complex of $P_X$, that is, the vertex set of $K_X$ is $[m]=\{1,2,\ldots, m\}$ and $\sigma=\{i_1, \ldots, i_k\} \in K_X$ if and only if $\bigcap_{i \in \sigma} F_i \neq \emptyset$; and
  \item $\lambda_X(\e_i)$ is congruent to the normal vector of $F_i$ modulo $2$ for $i=1, \ldots, m$, where $\e_i$ is the $i$th canonical vector in $\Z_2^m$; it is called the (mod $2$) \emph{characteristic function} of $X^\R$.
\end{itemize}
For an abstract simplicial complex $K$ on $[m]$, a linear map $\lambda \colon \Z_2^m \to \Z_2^n$ is said to satisfy the \emph{non-singularity condition} over $K$ if for each $\sigma=\{i_1, \ldots, i_k\} \in K$, $\lambda(\e_{i_1}), \ldots, \lambda(\e_{i_k})$ are linearly independent over $\Z_2$.
We remark that $\lambda_X$ satisfies the non-singularity condition over $K_X$.

From the viewpoint of topologists, the topology of $X^\R$ is more complicated than the topology of $X$.
For instance, $X$ is simply connected whereas $X^\R$ is never simply connected, and the integral cohomology ring of $X$ can be described beautifully as a quotient of the polynomial rings whereas the integral cohomology ring of $X^\R$ remains unknown.
Indeed, although the cohomology formula of a projective toric manifold $X$ has been well-established since the late 1970s due to Jurkiewicz \cite{Jurkiewicz1980} (and Danilov \cite{Danilov1978}), only partial results of cohomology of $X^\R$ have been obtained.
In general, the real locus of a complex variety has a rather more complicated topology than the original.
Thus, the computation of the cohomology of the real varieties in specific cases is always a good challenge to topologists (cf. \cite{Etingof-Henroques-Kamnitzer-Rains2010}).
Hence, numerous studies have attempted to compute the integral cohomology group of real toric varieties in the last decades.

Throughout this paper, $H^\ast(X;R)$ and $\tilde{H}^\ast(X;R)$ denote the (ordinary) cohomology and the reduced cohomology of $X$ with the coefficient $R$, respectively.
In 1985, $H^\ast(X^\R ; \Z_2)$ was computed by Jurkiewicz \cite{Jurkiewicz1985}.
In their unpublished paper \cite{ST2012} (or see \cite{Trevisan2012}), Suciu and Trevisan presented a formula that describes the additive structure of $H^\ast(X^\R;\Q)$, and the formula was confirmed and extended in \cite{Choi-Park2017_torsion} to the coefficient $\Z_q$, where $q$ is an odd integer.
The multiplicative structure of $H^\ast(X^\R;R)$ for $R$ is $\Q$ or $\Z_q$ ($q$ is odd) is also computed in \cite{Choi-Park2017_multiplicative}.

The main motivation of this paper is to compute the $2^k$-torsion ($k>1$) of $X^\R$ or, equivalently, to compute the additive structure of $H^\ast(X^\R; \Z_{2^k})$ for $k>1$.
Using the universal coefficient theorem, it enables us to derive the complete formula for the integral cohomology group of $X^\R$.

It is natural to consider a topological generalization of $X^\R$, and compute its cohomology.
One significant generalization of $X^\R$ is a small cover introduced in \cite{Davis-Januszkiewicz1991}. We note that $X^\R$ admits $(\R^\ast)^n$-action induced from the $(\C^\ast)^n$-action on $X$, where $\R^\ast = \R \setminus \{0\}$.
The action of $(\R^\ast)^n$ induces the action of $\Z_2^n \cong (S^0)^n = \{ \pm 1 \}^n \subset (\R^\ast)^n$, and $X^\R / \Z_2^n$ is combinatorially equivalent to $P_X$.
A \emph{small cover} is a smooth manifold that admits a locally standard $\Z_2^n$-action whose orbit space can be identified with an $n$-dimensional simple convex polytope $P$ with $m$ facets.
A small cover can also be classified by a pair of the boundary complex $K$ of $P$ and a characteristic function $\lambda \colon \Z_2^m \to \Z_2^n$ satisfying the non-singularity condition over $K$.
It should be noted that the boundary complex $K$ of a simple convex polytope $P$ is shellable \cite{Bruggesser-Mani1971}.

More generally, let $K$ be an abstract simplicial complex on $[m]$.
The \emph{real moment-angle complex} $\RZ_K$ of $K$ is the polyhedral product $(\underline{D^1},\underline{S^0})^K$ (the details are provided in Section~\ref{sec:2}) and it admits a $\Z_2^m$-action. For any linear map $\lambda \colon \Z_2^m \to \Z_2^n$, the kernel $\lambda$ is in $\Z_2^m$ and it acts on $\RZ_K$ as well.
The \emph{real toric space} $M^\R(K,\lambda)$ associated to a pair of $K$ and $\lambda$ is defined to be the quotient space $\RZ_K / \ker \lambda$.
It is known that a real toric manifold $X^\R$ is a real toric space $M^\R(K_X, \lambda_X)$ associated to $K_X$ and $\lambda_X$, and a small cover is also a real toric space associated to a polytopal simplicial complex $K$ and some characteristic function satisfying the non-singularity condition  over $K$.

Throughout the paper, we identify the power set of $[m]$ with $\Z_2^m$ in the standard way;
each element of $[m]$ corresponds to a non-zero coordinate of $\Z_2^m$.
For an element $\omega \in \Z_2^m$, $K_\omega$ is the full subcomplex of $K$ induced by the subset of $[m]$ associated to $\omega$.
Assume that $\lambda$ is represented by an $n\times m$ $\Z_2$-matrix $\Lambda$.
Then, the subspace spanned by the row vectors of $\Lambda$ in $\Z_2^m$ only depends on $\lambda$.
Hence, we denote by $\row \lambda$ the row space of $\Lambda$.

Among real toric spaces, in this paper, we mainly focus on the real toric spaces $Y$ associated to a pure shellable complex $K$ on $[m]$ whose dimension is $n-1$, and a characteristic function $\lambda \colon \Z_2^m \to \Z_2^n$ satisfying the non-singularity condition over $K$.
Then, $\ker \lambda \subset \Z_2^m$ acts on $\RZ_K$ freely, and there is no subgroup of $\Z_2^m$ that acts on $\RZ_K$ freely and contains $\ker \lambda$ as a proper subgroup.
Hence, $Y$ can be understood as the quotient spaces $\RZ_K / \Z_2^{k}$, where $K$ is a pure shellable complex and $\Z_2^k \subset \Z_2^m$ is a maximal free action on $\RZ_K$.

We firstly show that a shelling of $K$ gives a cell structure on $Y$ (up to homotopy) such that each dual cell (as a cellular chain) is a mod $2$ cocycle and all dual cells give the basis for $H^\ast(Y;\Z_2)$.
In the case when $K$ is polytopal, this cell structure specializes to the one given in \cite{Davis-Januszkiewicz1991} by a Morse function (see Remark~\ref{rem:smooth}).
In addition, with integer coefficients, we construct that the corresponding $\PL$ cellular cochain complex $(C^{\ast}(Y),d)$ of $Y$ is cochain-homotopy equivalent to a chain complex $(\ol{C}^\ast_\lambda,2d')$, where $(\ol{C}^\ast_\lambda,d')$ is a direct sum of the cellular cochain complexes of $K_\omega$ for $\omega \in \row \lambda$ (see Theorem~\ref{thm:main}).
This is done by a detailed comparison of the cellular cochains of $Y$ and that of its covering space $\RZ_K$, using \emph{transfer homomorphisms} at the cochain level (see Section~\ref{sec:6}).
Thus, we have the following theorem immediately.


\begin{theorem}\label{thm:integral_cohom_of_Y}
    Let $K$ be a pure shellable simplicial complex on $[m]$, the dimension of which is $n-1$, and let $\lambda \colon \Z_2^m \to \Z_2^n$ be a characteristic function satisfying the non-singularity condition over $K$.
    Let $Y = M^\R(K,\lambda)$ be the real toric space associated to $K$ and $\lambda$, and let $\ol{G}_\lambda^i$ be the group
	  $\bigoplus_{\omega\in \row \lambda}\wt{H}^i(K_{\omega};\Z )$,
	  $i\geq -1$,
	  and let $p$ be an odd prime, where $\wt{H}^{-1}(K_{\omega};\Z )$
		is non-trivial only when $\omega$ is a zero vector, which is
		infinite cyclic.
	  Then: 	
	  \begin{enumerate}
		\item the $\Z $-summands (respectively $\Z_{p^k}$-summands, $k\geq 1$)
		  of $\ol{G}_\lambda^i$ are in one-to-one correspondence
		  with that of $H^{i+1}(Y;\Z )$; and
	\item for $k\geq 1$,
	  the $\Z_{2^k}$-summands of $\ol{G}_\lambda^{i}$ are in one-to-one correspondence
	  with the $\Z_{2^{k+1}}$-summands of $H^{i+1}(Y;\Z )$.
	  \end{enumerate}
\end{theorem}

It should be noted that the result itself of (1) in the above main theorem is already known in \cite{ST2012, Choi-Park2017_torsion}, whereas the result of (2) is new.
However, our piecewise linear (PL) cell structure of $Y$ (in Section~\ref{sec:6}) provides a new (topological) proof of it, and it can provide more information such as Theorem~\ref{cor:bss} below.

The most important corollary of Theorem~\ref{thm:integral_cohom_of_Y} is that it allows us to compute the integral cohomology group $H^\ast(Y;\Z)$ of $Y$ in terms of $K$ and $\lambda$.
The formula of the $\Z_2$-cohomology ring of $Y$ has been established by
Davis and Januszkiewicz \cite[Theorem~5.12]{Davis-Januszkiewicz1991} (or Proposition~\ref{prop:DJ} below).
Remarkably, the additive structure of $H^\ast(Y;\Z_2)$ only depends on $K$. The $i$th $\Z_2$-Betti number of $Y$ is equal to the $i$th component of the $h$-vector of $K$.
Therefore, by combining this with the main theorem, according to the universal coefficient theorem for the cohomology, one can compute the integral cohomology group of $Y$ in terms of $K$ and $\lambda$, as well as the real locus of a projective toric manifold and a small cover.

\begin{corollary} \label{cor:cohom_of_Y}
    The integral cohomology group $H^\ast(Y;\Z)$ of $Y$ is completely determined by the reduced cohomology group of $K_\omega$ (for $\omega \in \row \lambda$) and the $h$-vector of $K$.
\end{corollary}

\begin{remark}
 Even if a toric manifold $X$ is non-projective, the corresponding real toric manifold $X^\R$ becomes a real toric space, namely, it can be represented by a pair of $K_X$ and $\lambda_X$.
In this case, however, $K_X$ is not necessarily polytopal (see \cite{Suyama2014}), whereas it is always star-shaped.
A real toric space associated to a pair of a star-shaped simplicial complex $K$ on $[m]$ whose dimension is $n-1$ and a characteristic map $\lambda \colon \Z_2^m \to \Z_2^n$ satisfying the non-singularity condition is known as a \emph{topological real toric manifold} introduced in \cite{Ishida-Fukukawa-Masuda2013}.
Hence, a topological real toric manifold is a topological generalization of real toric manifold.
As mentioned in \cite[p.~279]{Ziegler1995book}, there is an old question that asks whether or not every star-shaped complex is shellable. If the answer to the question is affirmative in general, then our formula in Theorem~\ref{thm:integral_cohom_of_Y} would cover all topological real toric manifolds, as well as all ordinary real toric manifolds.
\end{remark}

Theorem~\ref{thm:integral_cohom_of_Y} provides many interesting corollaries on the torsion elements of real toric manifolds and small covers.
It should be mentioned that, for any integer $\ell \geq 2$, there is a real toric manifold $X^\R$ (which is the real locus of projective toric manifold $X$) and an element $\omega \in \row \lambda_X$ such that $K_\omega$ has an $\ell$-torsion element in the cohomology.
It was proved in \cite[Theorem~5.10]{Choi-Park2017_torsion} only when $\ell$ is an odd number, but the proof, essentially, can be applicable for arbitrary integers.
Therefore, we have the following important corollaries.
\begin{corollary}
    For any integer $\ell>1$, there are infinitely many real toric manifolds and small covers whose integral cohomology groups have an $\ell$-torsion element.
\end{corollary}

We also remark that Theorem~\ref{thm:integral_cohom_of_Y} provides a simple criterion for $Y$ to have $\ell(>3)$-torsion elements.
  \begin{corollary} \label{cor:at_most_2_torsion}
  The following statements are equivalent:
  \begin{enumerate}
	\item
	  $H^\ast (Y;\Z )$ is either torsion-free or has only two-torsion elements;
	\item $H^\ast (K_{\omega};\Z )$ is  torsion-free, for all $\omega\in \row \lambda$.
  \end{enumerate}
     \end{corollary}

If $K$ is the boundary complex of a simple $n$-polytope, each $K_\omega$ must be homotopy equivalent to a disjoint union of finitely many punctured spheres of dimension $n-1$.
By Alexander duality, one can see that a simplicial complex with torsion in cohomology cannot be a subcomplex in any triangulation of $S^{n-1}$ for $n \leq 4$.
Therefore, $K_\omega$ is torsion-free for arbitrary $\omega \subset [m]$ and, hence, the following corollary follows from Corollary~\ref{cor:at_most_2_torsion}.

\begin{corollary}
A small cover of dimension at most four has at most $2$-torsion in the cohomology group.
\end{corollary}

More precisely, one can compute the complete formula for the cohomology group of three- and four-dimensional small cover $Y$ as in Corollary~\ref{cor:cohom_of_Y}.
Let $Y^3$ be a three-dimensional small cover $Y^3 = M^\R(K,\lambda)$, where $K$ has $m$ vertices.
We call the rank of $\tilde{H}^i(X;\Q)$ the $i$th \emph{reduced Betti number} of $X$.
Denote by $b$ the sum of the zeroth reduced Betti number of $K_\omega$ for $\omega \in \row \lambda$.
Since there are $2^3 - 1 = 7$ non-zero elements in $\row \lambda$, $b + 7$ is the sum of the number of connected components of $K_\omega$ for non-zero $\omega$s.
Hence, $b$ can be computed combinatorially.

We have the following table (compare with \cite[Proposition~VII.1.3]{Trevisan2012}):

\begin{center}
\begin{tabular}{|c|c|c|}
  \hline
  $H^i(Y^3;\Z)$ & Orientable & Non-orientable \\ \hline
  $i=0$ & $\Z$ & $\Z$ \\
  $i=1$ & $\Z^b$ & $\Z^b$ \\
  $i=2$ & $\Z^b \oplus \Z_2^{m-3-b}$ & $\Z^{b-1} \oplus \Z_2^{m-3-b}$ \\
  $i=3$ & $\Z$ & $\Z_2$ \\
  \hline
\end{tabular}
\end{center}

Let $Y^4$ be a four-dimensional small cover $Y^4 = M^\R(K,\lambda)$, where $K$ has $m$ vertices.
Denote by $b$, $c$, and $d$ the sum of the zeroth, first, and second reduced Betti number of $K_\omega$ for $\omega \in \row \lambda$, respectively.
Then, we have the following table:

\begin{center}
\begin{tabular}{|c|c|c|}
  \hline
  $H^i(Y^4;\Z)$ & Orientable & Non-orientable \\ \hline
  $i=0$ & $\Z$ & $\Z$ \\
  $i=1$ & $\Z^b$ & $\Z^b$ \\
  $i=2$ & $\Z^c \oplus \Z_2^{m-4-b}$ & $\Z^{c} \oplus \Z_2^{m-4-b}$ \\
  $i=3$ & $\Z^b \oplus \Z_2^{m-4-b}$ & $\Z^{d} \oplus \Z_2^{m-5-d}$ \\
  $i=4$ & $\Z$ & $\Z_2$ \\
  \hline
\end{tabular}
\end{center}

In hyperbolic geometry, there is one important class of hyperbolic $3$-spaces, called \emph{hyperbolic $3$-manifolds of L\"{o}bell type}, introduced by Vesnin \cite{Vesnin1987}.
Interestingly, a hyperbolic $3$-manifold of L\"{o}bell type is a small cover.
In general, a three-dimensional small cover $Y = M^\R(K,\lambda)$ is hyperbolic if and only if $K$ is flag and it has no chordless $4$-cycle.
(See the summary in \cite[Section~2.5]{BEMPP}.)
It is worth emphasizing that our formula provides a complete integral cohomology formula for hyperbolic three-dimensional small covers.

On the other hand, we note that there is a classical method to recover the $2^k$-torsion elements from $H^\ast(X;\Z_2)$, called the \emph{Bockstein spectral sequence}.
For another important result, we calculate explicitly the (higher) mod $2$ Bockstein homomorphisms on $H^\ast(Y;\Z_2)$. See Section~\ref{sec:bss}.
		
\begin{theorem}\label{cor:bss}
Let $K$ be a pure shellable simplicial complex on $[m]$ whose dimension is $n-1$ and let $\lambda \colon \Z_2^m \to \Z_2^n$ be a characteristic function satisfying the non-singularity condition over $K$.
Let $Y = M^\R(K,\lambda)$ be the real toric space associated to $K$ and $\lambda$.
		  In the corresponding
		  Bockstein spectral sequences, $E^i_{k+1}(Y)$ is isomorphic
		  to $\oplus_{\omega\in \row \lambda}E^{i-1}_k(K_{\omega})$,
		  for all $i,k\geq 1$. In particular,
		  \[      E^i_{2}(Y)\cong \bigoplus_{\omega\in \row \lambda}
			\wt{H}^{i-1}(K_{\omega};\Z_2), \quad i\geq 1.       \]
		\end{theorem}

\section{Real toric spaces and shellability}\label{sec:2}
An (abstract) \emph{simplicial complex} on a finite set $V$ is a collection of subsets of $V$ satisfying:
\begin{enumerate}
  \item if $v \in V$, then $\{v \} \in K$; and
  \item if $\sigma \in K$ and $\tau \subset \sigma$, then $\tau \in K$.
\end{enumerate}
Each element $\sigma \in K$ is called a \emph{simplex} or a \emph{face} of $K$, and each element of $V$ is called a \emph{vertex} of $K$.
The \emph{dimension} of $\sigma$ is defined by $\dim \sigma = \card{\sigma} -1$, where $\card \sigma$ denotes the cardinality of $\sigma$.
The \emph{dimension} of $K$ is defined by $\dim K = \max \{ \dim \sigma \mid \sigma \in K\}$.

If an abstract simplicial complex $K$ is isomorphic with the vertex scheme of the simplicial complex $K'$, then $K'$ is said to be a \emph{geometric realization} of $K$, and is uniquely determined up to a linear isomorphism.
Throughout this paper, we do not distinguish a simplicial complex and its geometric realization.
A simplicial complex $K$ is said to be a \emph{simplicial sphere} of dimension $n-1$ if it can be realized by a triangulation of the $(n-1)$-dimensional sphere.
A simplicial sphere $K$ of dimension $n-1$ is said to be \emph{polytopal} if it can be realized by the boundary complex of a convex simple polytope of dimension $n$.

Let $K$ be a simplicial complex on $[m]=\{1,\ldots, m\}$ and let $(X,A)$ be a pair of topological spaces.
The \emph{polyhedral product} $(\underline{X}, \underline{A})^K$ is defined as follows:
$$
    (\underline{X}, \underline{A})^K := \bigcup_{\sigma \in K} \{(x_1, \ldots, x_m ) \in X^m \mid x_i \in A \text{ if } i \notin \sigma \}.
$$
In particular, for an interval $D^1 = [0,1]$ and its boundary $S^0 = \{0,1\}$, the polyhedral product $(\underline{D^1}, \underline{S^0})^K$ is known as the \emph{real moment-angle complex} of $K$, and it is simply denoted by $\RZ_K$.
It should be mentioned that there is a canonical $\Z_2^m$-action on $\RZ_K$ that comes from the $\Z_2$-action on the pair $(D^1, S^0)$.
It is well-known that $\RZ_K$ is a topological manifold if $K$ is a simplicial sphere, and $\RZ_K$ is a smooth manifold if $K$ is polytopal (cf. \cite[Section~2]{Cai2017}).

For $\ell<m$, let $\lambda \colon \Z_2^m \to \Z_2^\ell$ be a linear map. Then, the kernel $\ker \lambda \subset \Z_2^m$ also acts on $\RZ_K$.
The \emph{real toric space} $M^\R(K,\lambda)$ associated to $K$ and $\lambda$ is defined to be
$$
    M^\R(K,\lambda) := \RZ_K / \ker \lambda.
$$
It should be noted that $\ker \lambda$ acts on $\RZ_K$ freely if $\lambda$ satisfies the \emph{non-singularity condition} over $K$; if for each $\sigma=\{i_1, \ldots, i_k\} \in K$, $\lambda(\e_{i_1}), \ldots, \lambda(\e_{i_k})$ are linearly independent over $\Z_2$
(cf. \cite[Lemma~3.1]{Choi-Kaji-Theriault2017}).
Therefore, under the assumption that $\lambda$ satisfies the non-singularity condition over $K$,
$M^\R(K,\lambda)$ is a topological manifold (respectively, smooth manifold) if $K$ is a simplicial sphere (respectively, $K$ is polytopal) as well.

Hence, we are usually interested in the case when $K$ is a simplicial (or polytopal) sphere of dimension $n-1$ and $\lambda \colon \Z_2^m \to \Z_2^\ell$ satisfies the non-singularity condition over $K$.
As the non-singularity implies $\ell \geq n$, we are particularly interested in the case when $\ell =n$.
In this case, $\lambda$ is called a \emph{characteristic function} of $M^\R(K,\lambda)$.
Real toric manifolds and small covers introduced in the introduction provide such examples.
Indeed, if $K$ and $\lambda$ came from the normal fan of a simple convex polytope, then the corresponding real toric space $Y = M^\R(K,\lambda)$ is a real toric manifold, and if $K$ is polytopal and $\ell=n$, then $Y$ is a small cover.

The non-singularity condition has an algebraic interpretation.
Let $K$ be an abstract simplicial complex on $[m]=\{1,\ldots, m\}$.
In the ring $\Z_2[x_1,\ldots,x_m]$ of polynomials with $\deg x_i=1$ for $i=1, \ldots, m$, we consider the ideal $I_K$ generated by square-free monomials $x_{i_1} \cdots x_{i_k}$ for all $\{i_1, \ldots, i_k\} \not\in K$. The quotient $\Z_2[x_1, \ldots, x_m]/ I_K$ is called the \emph{Stanley--Reisner ring} or the \emph{face ring} of $K$ over $\Z_2$, and it is denoted by $\Z_2[K]$.

Let $K$ be a simplicial complex whose Krull dimension is $n$, that is, there exist homogeneous polynomials $\theta_1,\ldots,\theta_n$
that are algebraically independent,
such that $\Z_2[K]$ is a finitely
generated
$\Z_2[\theta_1,\ldots,\theta_n]$-module.
Such a sequence $\theta_1,\ldots,\theta_n$ is called a \emph{homogeneous system of parameters} (h.s.o.p.).
If, in addition, $\Z_2[K]$ is a
free $\Z_2[\theta_1,\ldots,\theta_n]$-module,
then $\Z_2[K]$ is \emph{Cohen--Macaulay}, and
we say that $K$ is a Cohen--Macaulay complex.
By a result of Reisner \cite{Reisner1976},
$\Z_2[K]$ is Cohen--Macaulay if and only if
$\widetilde{H}_i(K;\Z_2)=0$ for all $i<n-1$, and
$\widetilde{H}_i(\Link(\sigma,K);\Z_2)=0$
for all $i<\dim \Link(\sigma,K)$.

If all $\theta_i$ are linear, then it is called a \emph{linear system of parameters} (l.s.o.p.) of $\Z_2[K]$.
The non-singularity of $\lambda$ implies that $\Z_2[K]$ has an l.s.o.p., and vice versa.
Assume that $\lambda$ can be represented by an $n \times m$ $\Z_2$-matrix $\Lambda$ whose $i$th row vector is $\lambda_i \in \Z_2^m$.
We recall that in this paper, we regard an element of $\Z_2^m$ as a subset of $[m]$ by standard identification.
We write
$$
  l_{\lambda_i}=
  \sum_{k\in\lambda_i}x_k, \quad \text{for } i=1,\ldots,n.
$$

The following criterion is useful.
\begin{proposition}[{Cf.~\cite[pp.~81--82]{Stanley1996book}}]\label{prop:Stan1}
The following
statements are equivalent:
\begin{enumerate}
\item $\lambda\colon\Z_2^m\to\Z_2^n$ is
	a characteristic function that satisfies the non-singularity condition over $K$;
  \item $l_{\lambda_1},\ldots,l_{\lambda_n}$
  is an \emph{l.s.o.p.} of $\Z_2[K]$
  associated to $\lambda$.
 \end{enumerate}
\end{proposition}

Interestingly, the Stanley--Reisner ring $\Z_2[K]$ and its quotient by l.s.o.p. are realized by the $\Z_2$-cohomology ring of some specific topological spaces corresponding to the real toric space.
We note that $Y=M^\R(K,\lambda)$ has the $\Z_2^n$-action inherited from the $\Z_2^m$-action of $\RZ_K$.
Now, let us consider the Borel construction $BK=E\Z_2^n \times_{\Z_2^n}Y$ of $Y=M^\R(K,\lambda)$ with respect to the $\Z_2^n$-action.
Then, we have a fiber bundle
\begin{equation} \label{eq:fib}
\xymatrix{
    Y \ar[r] & BK  \ar[r] &  B\Z_2^n.
}
\end{equation}
	Let $(\RP^{\infty},\ast)$ be the infinite real
	  projective space with a base point.
	  \begin{proposition}[{Cf.~\cite{Davis-Januszkiewicz1991}}]\label{prop:BP}
  	Up to homotopy, $BK = (\underline{\RP^{\infty}}, \underline{\ast})^K$.
    Moreover, the $\mathrm{mod}$ $2$ cohomology of $BK$ is isomorphic
 to the Stanley--Reisner ring of $K$, namely
 \[ H^\ast(BK;\Z_2)\cong \Z_2[x_1,\ldots,x_m]/I_K
 \quad (\deg x_i=1 \text{ for }i=1,\ldots,m), \]
 where the Stanley--Reisner ideal $I_K$ is generated by square-free monomials of the
 form $x_{\tau}=\prod_{i\in\tau}x_i$, $\tau\not\in K$.
    \end{proposition}

The quotient ring
$\Z_2[K]/(l_{\lambda_1},\ldots,l_{\lambda_n})$ has
the following topological interpretation.
\begin{proposition}[{Cf.~\cite[Theorem 5.12]{Davis-Januszkiewicz1991}}]\label{prop:DJ}
When $K$ is a Cohen--Macaulay complex admitting a characteristic
function $\lambda$, the inclusion $ Y\to BK$ in the
fibration \eqref{eq:fib} induces
a surjection in $\mathrm{mod}$ $2$ cohomology,
and we have an isomorphism
\begin{equation}
  H^*(Y;\Z_2)\cong
  \Z_2[K]/(l_{\lambda_1},\ldots,l_{\lambda_n})
  \label{eq:DJ}
\end{equation}
of graded rings.
   \end{proposition}

A simplicial complex is said to be \emph{pure} if its maximal simplices  (also simply called \emph{facets}) all have the same dimension.
For a pure simplicial complex $K$ of dimension $n-1$, an ordering $\sigma_1, \ldots, \sigma_s$ of the facets of $K$ is called a \emph{shelling} if the complex $B_j := \left( \bigcup_{i=1}^{j-1} \sigma_i \right) \cap \sigma_{j}$ is pure and $(n-2)$-dimensional for all $j=2, \ldots, s$.
If $K$ admits a shelling, then $K$ is said to be \emph{shellable}.
It is a well-known fact that a pure shellable simplicial complex is Cohen--Macaulay.
For a shellable simplicial complex $K$, there is a natural filtration
$\sigma_1=K_1\subset K_2\subset\cdots\subset K_s=K$ such that
	  $K_{j+1}=K_j \cup \sigma_{j+1}$.
For each $j$, we have a minimal face $r(\sigma_{j+1})$ among all faces of
	  $K_{j+1}$ that are not contained in $K_{j}$, called the
	  \emph{restriction} of $\sigma_{j+1}$, which
	  is unique. Formally, we set
$$
    r(\sigma_1)=\emptyset, \quad x_{r(\sigma_1)}=1,
$$
	  where $x_{\tau}=\prod_{i\in\tau}x_i$ for $\tau\subset [m]$.	

Throughout this section, we assume that $K$ is a pure shellable simplicial complex on $[m]$ of dimension $n-1$ such that $\sigma_1,\ldots,\sigma_s$ is a shelling of $K$ and its filtration is $\sigma_1=K_1\subset\cdots\subset K_s=K$.
We also assume that $\lambda \colon \Z_2^m \to \Z_2^n$ is a  characteristic function satisfying the non-singularity condition over $K$  and $Y$ is the real toric space associated to $K$ and $\lambda$.

The following basis for $H^\ast (Y;\Z_2)$ is convenient.
\begin{proposition}[{Cf.~\cite[pp.~82--83]{Stanley1996book}}]\label{prop:basis}
		Let $\sigma_1, \ldots, \sigma_s$ be a shelling of $K$
		with characteristic function $\lambda$.
		Then $\Z_2[K]$ is a free
		$\Z_2[l_{\lambda_1},\ldots,l_{\lambda_n}]$-module
		with basis $\{x_{r(\sigma_j)}\}_{j=1}^s$.
  \end{proposition}

   Recall that for $0\leq k\leq n$, an $n$-dimensional $k$-handle
   $W$ on a $\PL$
   manifold $X$ is a copy of $(D^1)^k\times (D^1)^{n-k}$,
   attached to the boundary $\partial X$ via a $\PL$ embedding
  $f\colon \partial (D^1)^k\times (D^1)^{n-k} \to \partial X$, where $D^1 =[0,1]$ is an interval.
The number $k$ is called the \emph{index} of the handle.
 It turns out that the union $X\cup_{f}W$ is
 again a $\PL$ $n$-manifold (cf.~\cite[Chapter 6]{Rourke-Sanderson1972book}).

For $\sigma\subset [m]$, let
$I_{\sigma}\subset (D^1)^m$ be
the cube whose $i$th component is either $D^1$ if $i\in\sigma$,
or $\{1\}$ otherwise, and $B_{\sigma} = (\underline{D^1}, \underline{S^0})^{2^\sigma }$.
Clearly $\dim I_{\sigma}=\card(\sigma)$.
We set $\RZ_1=B_{\sigma_1}$.
As $\lambda|_{\sigma_1}$ is non-degenerate,
$\RZ_1$ coincides with the orbit of $I_{\sigma_1}$ under
the action of $\ker \lambda$, and
the orbit map $\pi$ maps $I_{\sigma_1}$ homeomorphically onto $Y_1$.
For $j=1,\ldots,s-1$, $\RZ_{j+1}$ is
obtained from $\RZ_j$ by attaching
$B_{\sigma_{j+1}}$, a disjoint union of $2^{m-n}$ cubes
of dimension $n$.

Consider the cube
\begin{equation}\label{eq:cube}
  I_{\sigma_{j+1}}=I_{r(\sigma_{j+1})}\times
  I_{\sigma_{j+1}\setminus r(\sigma_{j+1})}
\end{equation}
  in $\RZ_{j+1}$.
Then, the non-degeneracy of $\lambda|_{\sigma_{j+1}}$ implies
  that $B_{\sigma_{j+1}}$ is the orbit of $I^n_{\sigma_{j+1}}$
  under $\ker\lambda$, which are attached along
  the orbit of $\partial I_{r(\sigma_{j+1})}\times  I_{\sigma_{j+1}\setminus r(\sigma_{j+1})}$.
  Moreover, this attachment is along the topological
  boundary $\partial \RZ_j$
  if and only if $\sigma_{j+1}$ is attached along the topological
  boundary $\partial K_j$. The orbit map $\pi$ maps $I_{\sigma_{j+1}}$
  homeomorphically onto its image in
  $M_{j+1}$, which is a $\PL$ embedding into
  $M_j$
  when restricted to $\partial I_{r(\sigma_{j+1})}\times
  I_{\sigma_{j+1}\setminus r(\sigma_{j+1})}$.

Suppose the cardinality
of $r(\sigma_j)$ is $k_j$.
    The next proposition follows from an induction on $j=1,\ldots,s-1$
	(see Example~\ref{exm:1} for an illustration).

\begin{proposition}\label{prop:dec}
The filtration $Y_1\subset\cdots\subset Y_s$ gives a $\PL$
	  handle decomposition of
	  $Y$, where $Y_1=(D^1)^n$, $Y_s=Y$ and $Y_{j+1}$ is obtained from $Y_j$ by attaching
	  a handle of index $k_{j+1}$, provided that
	  $K_{j+1}$ is obtained by attaching $\sigma_{j+1}$ along
	 $\partial K_j$ for $j=1,\ldots,s-1$.
	  In general,
	  $Y_{j+1}$ is obtained from $Y_j$ by attaching a cell of dimension
	  $k_{j+1}$, up to homotopy.
\end{proposition}

	\begin{remark}\label{rem:smooth}
	 When $K$ is polytopal, that is, $K$ is the boundary complex of
	 a convex polytope $P \subset\R^n$, one can choose a generic
	 linear function $f\colon\R^n\to\R$ such
	 that the gradient of $f$ gives an ordering $v_1,\ldots,v_s$
	 of vertices of $P$, as well
	 as an orientation for each edge of $P$.
	 As a vertex $v$ of $P$ corresponds to a facet $\sigma$ of $K$, and
	 the $n$ edges around $v$ correspond to the $n$ facets of $\sigma$, 	
	 $f$ gives a shelling $\sigma_1,\ldots,\sigma_s$ of $K$.
    In this case, we can use the smooth structure of $P\subset\R^n$
	 to smooth the $\PL$ handles of $Y$,
	 which coincides with that
	 given in \cite[pp.~431--432]{Davis-Januszkiewicz1991}.
    Let $p\colon Y\to P$ be the orbit map associated to
	 the $\Z_2^n$-action on $Y$.
    Then, it can be checked directly that $\pi\colon\RZ_K \to Y$
	 maps the interior of $I_{\sigma_i}$ homeomorphically
	 onto a neighborhood
	 $p^{-1}(U_i)$, $U_i$ a neighborhood of $v_i\in P$,
     where the closure of $p^{-1}(U_i)$ is the smooth handle
	 from the Morse function.
\end{remark}

   \begin{remark}
	 It is a known fact in $\PL$ topology that, 	
	 if each attaching of $\sigma_{j+1}$ is along
	 $\partial K_j$ (in the shelling),
	  then the geometric realization of $K$ is either a $\PL$ sphere or 	
	  a $\PL$ disk of dimension $n-1$,
	  depending on whether the attaching of the last facet $\sigma_s$
	  is along its whole boundary or not.
	  Therefore, when $K$ is shellable (and admits a characteristic
	  function), $\RZ_K$ (respectively, $Y$) is a closed
	  $\PL$ $n$-manifold if and only if $K$ is a $\PL$
	  $(n-1)$-sphere. It should be noted that, in fact, the assumption of shellability
		is not essential. See \cite[Theorem 2.3]{Cai2017}.
	\end{remark}

Let $l_{\lambda_1}|_j,\ldots,l_{\lambda_n}|_j$ be the
	  image under $\Z_2[K]\to\Z_2[K_j]$
	  of $l_{\lambda_1},\ldots,l_{\lambda_n}$, the
	  l.s.o.p.\ associated to the characteristic function $\lambda$
	  (see Proposition \ref{prop:Stan1}),
	   which is induced
	  by the inclusion $K_j\to K$.

\begin{proposition}\label{prop:DJSta}
		  The ring $\Z_2[K_j]$ is a free
		  $\Z_2[l_{\lambda_1}|_j,\ldots,
			l_{\lambda_n}|_j]$-module with basis $\{x_{r(\sigma_t)}\}_{t\leq j}$,
			and $H^*(Y_j;\Z_2)$ is isomorphic to
			\[ \Z_2[K_j]/(l_{\lambda_1}|_j,\ldots,
			l_{\lambda_n}|_j)\cong
			\Z_2[K]/(l_{\lambda_1},\ldots,l_{\lambda_n},
			x_{r(\sigma_{j+1})},\ldots,
			x_{r(\sigma_{s})})\]
			as graded rings.
\end{proposition}

\begin{proof}
		  Note that
		  $K_j$ is shellable, thus is Cohen--Macaulay.
      The first statement follows from Proposition~\ref{prop:basis}.
		The second follows from
		  Proposition~\ref{prop:DJ}, together with the observation
		  that
		  $\Z_2[K]\to\Z_2[K_j]$ coincides
		  with $H^*(BK;\Z_2)\to H^*(BK_j;\Z_2)$,
		  which is induced from the inclusion $BK_j\to BK$
		  (see Proposition~\ref{prop:BP}).
		\end{proof}

\section{The cochain complex of $Y$}\label{sec:6}
Let $K$ be a pure shellable simplicial complex on $[m]$ whose dimension is $n-1$ and let $\lambda \colon \Z_2^m \to \Z_2^n$ be a characteristic function satisfying the non-singularity condition over $K$.
Let $Y = M^\R(K,\lambda)$ be the real toric space associated to $K$ and $\lambda$.
This section is devoted to the construction of the cochain complexes $(C^\ast(Y),d)$ of $Y$.


\subsection{Oriented cellular chain complex of $Y$}
An orientation of a $k$-cell $e$ in a $CW$-complex $X$ is a chosen generator in $H_k(e,\partial e)$, where $\partial e$ is the topological boundary of $e$.
An oriented cell $e$ will be denoted by $[e]$, and the boundary
$\partial [e]$ of an oriented $k$-cell $[e]$ is a $\Z$-linear sum of oriented $(k-1)$-cells.
Denote by $(C_\ast(X),\partial)$ the cellular chain complex with respect to the cell decomposition of $X$, where we use the notation $\partial$ again for the boundary operator in chain complexes.

We omit the proof of the following standard lemma.

\begin{lemma}\label{lem:cell}
Suppose $X$ is a $CW$-complex with a (left) action of a finite group $G$
  such that
\begin{equation}\label{eq:cell_in_action}
\text{for every $g\in G$ and every cell $e\subset X$,
  $g \cdot e$ is again a cell of $X$}.
\end{equation}
	Then the cell structure of $X$ descends to that of the orbit space $X/G$.
\end{lemma}

If an action of $G$ on $X$ satisfies \eqref{eq:cell_in_action}, then each transformation $g\colon X \to X$ is cellular.
Thus, it induces a chain map $g_\ast \colon (C_\ast(X), \partial) \to (C_\ast(X), \partial)$, namely, if $g_\ast([e])=\varepsilon [g \cdot e]$, then
  $g_\ast (\partial [e])=\varepsilon \partial [g\cdot e]$, $\varepsilon = \pm 1$, where $g_\ast \colon H_\ast(e,\partial e)\to H_\ast (g\cdot e,\partial (g\cdot e))$ is induced by $g$.
  Let $\pi\colon X\to X/G$ be the orbit map. As $\pi$ is cellular, it also induces a chain map
  $\pi_\ast\colon (C_\ast (X),\partial)\to (C_\ast (X/G),\partial)$,
  i.e.,
  \begin{equation} 	\label{eq:bdM0}
	\partial\pi_\ast ([e])=\pi_\ast (\partial [e]).
  \end{equation}

Regarding the interval $D^1 = [0,1]$ as the $CW$-complex consisting of two $0$-cells $0$, $1$ and one $1$-cell $\underline{01}$, the chain complex $C_\ast(D_i^1)$ is the graded $\Z$-module $\langle [0], [1], [\underline{01}] \rangle$ such that $\partial [\underline{01}] = [1] - [0]$ and $\partial [1] = \partial [0]= 0$.

We note that the $m$-cube $(D^1)^m$ has a natural $CW$-structure coming from the Cartesian product operation.
More precisely, let $D_i^1 \cong [0,1]$ be the $i$th factor of $(D^1)^m = D^1_1 \times \cdots \times D^1_m$ that is a $CW$-complex with two $0$-cells $0_i$, $1_i$, and one $1$-cell $\underline{01}_i$.
Then every cell of $(D^1)^m$ is given as
\[
    e=e_1\times\cdots\times e_m, \quad e_i=0_i,\ 1_i,  \text{ or }\underline{01}_i.
\]
Each cell $e$ is called a \emph{cubical cell}.
For a cubical cell $e$, let us define
  $$
	\sigma_{e}=\{i\mid e_i= \underline{01}_i\}, \quad \tau_e^+=\{i\mid e_i=1_i\},
	\quad \text{and} \quad \tau_e^-=\{i\mid e_i=0_i\},
  $$
which are disjoint subsets with their union $[m]$.
For given $i\in\sigma_e$, the $i$th \emph{front face} $\partial_i^+e$ and the $i$th \emph{back face} $\partial_i^-e$ are defined as
$$
    \partial_i^+e = e_1 \times \cdots \times e_{i-1} \times 1_i \times e_{i+1} \times \cdots \times e_m
$$
and
$$
    \partial_i^-e = e_1 \times \cdots \times e_{i-1} \times 0_i \times e_{i+1} \times \cdots \times e_m.
$$

By the Eilenberg--Zilber theorem \cite{Eilenberg-Zilber1953}, the oriented cellular chain $[e]=[e_1] \otimes \cdots \otimes [e_m]$ is endowed with the boundary operator
\begin{align}
  \partial [e] &=\sum_{i=1}^m (-1)^{(\sigma_e,i)}[e_1] \otimes\ldots\otimes [e_{i-1}]\otimes
  \partial [e_i] \otimes [e_{i+1}] \otimes\ldots\otimes [e_m] \nonumber
  \\
  &=
  \sum_{i\in\sigma_e}(-1)^{(\sigma_e,i)}([\partial_i^+e]-[\partial_i^-e]),
  \label{eq:bd}
\end{align}
where $(\sigma_e,i)=\mathrm{card}(\{j\in\sigma_e\mid j<i\})$.
In what follows, the notation $[e]$ means that the cubical cell $e$ is endowed with the orientation from the Eilenberg--Zilber theorem.

We see that the real moment-angle complex $\RZ_K \subset (D^1)^m$ has a cell decomposition with cubical cells $e$ such that $\sigma_e\in K$.
Let $(C_\ast (\RZ_K),\partial)$ be the oriented cellular chain complex above, where $\partial$ follows \eqref{eq:bd}.
  It turns out that  the homology of $(C_\ast (\RZ_K),\partial)$  gives the cellular homology of $\RZ_K$   (cf. \cite[Theorem 3.1]{Cai2017}).

  With this decomposition, as the action of $\ker \lambda$ on $\RZ_K$ satisfies \eqref{eq:cell_in_action}, it gives a cell structure on $Y$ by Lemma~\ref{lem:cell}.
More explicitly, for $g=(g_1, g_2, \ldots, g_m) \in \ker\lambda$, we have
  \begin{equation} \label{eq:ref}
	g_\ast ([e]) =\bigotimes_{i=1}^m (g_i)_\ast [e_i] =(-1)^{(\sigma_e,g)}[g\cdot e],
  \end{equation}
  where $(\sigma_e,g)=\card (\{i\in\sigma_e\mid g_i=1\})$ and $g \cdot e$ is again cubical.
In addition, if $g_i=1$, then
  $(g_i)_\ast ([\underline{01}_i])=-[\underline{01}_i]$ reversing the orientation,
  $(g_i)_\ast ([1_i])=[0_i]$, and $(g_i)_\ast ([0_i])=[1_i]$.

To describe the chain complex $(C_\ast (Y),\partial)$ in terms of cubical cells, we introduce the notion of a canonical cubical cell, which is the key idea of our construction.
Let $\fF_K = \{\sigma_1, \ldots, \sigma_s\}$ be the set of maximal faces (simply, \emph{facets}) of $K$ and a shelling $\sigma_1,\ldots,\sigma_s$ of $K$ is given in this order.
We consider the map $f\colon K\to \fF_K$ such that $f(\sigma)$ is the first facet in the sequence that contains $\sigma$.
A cubical cell $e$ is said to be \emph{canonical} if $\tau_e^-\subset f(\sigma_e)$.


As we mentioned before, we identify $\Z_2^m$ and the power set of $[m]$ in the standard way.

\begin{lemma}\label{lem:g}
Let $\pi\colon\RZ_K \to Y$ be the orbit map.
For each cubical cell $e$ of $\RZ_K$, there exists a unique $g_e\in \ker \lambda$ such that $g_e \cdot e$ is canonical.
\end{lemma}

\begin{proof}
	Consider $\tau=\tau_e^-\setminus f(\sigma_e)$.
    One can easily see that $g_e \cdot e$ is canonical if $g_e \setminus f(\sigma_e)=\tau$.
    Without loss of generality, suppose $f(\sigma_e)=\{1,2,\ldots,n\}$. 	
    The non-degeneracy of $\lambda|_{f(\sigma_e)}$ implies that
	$\ker \lambda$ can be represented as the column space of an $m\times (m-n)$-matrix of the form
    $$
    \bordermatrix{%
      &1	 &2     &\cdots  & m-n  \cr
1	  &\ast 	 & \ast     &\cdots	 & \ast     \cr
2     &\ast 	 & \ast     &\cdots	 & \ast     \cr
\vdots&\vdots   &\vdots     &\ddots	 &\vdots \cr
n	  &\ast      & \ast     &\cdots	 & \ast \cr
n+1	  &1      & 0     &\cdots	 & 0 \cr
n+2	  &0      & 1     &\cdots	 & 0 \cr
\vdots&\vdots    & \vdots     &\ddots	 & \vdots \cr
m     &0      & 0     &\cdots	 & 1 \cr
  }.
    $$
%
%
    Let $v_i$ be the $i$th column vector of $\ker \lambda$.
    The element $g_e:=\sum_{i\in\tau}v_{i-n}$ satisfies $g_e \setminus f(\sigma_e)=\tau$, which is unique.
\end{proof}

By Lemma~\ref{lem:g}, each cell of $Y$ is the image under $\pi$ of a canonical one.
For each canonical cell $e$, we assume that the orientation of $\pi e$ is given by $\pi_\ast ([e])$, that is, $[\pi e] = \pi_\ast ([e])$.
Then, by  \eqref{eq:bdM0} and \eqref{eq:bd}, the boundary operator in $(C_\ast (Y),\partial)$ follows
$$
	  \partial[\pi e]=\pi_\ast \partial [e] =\sum_{i\in\sigma_e}(-1)^{(\sigma_e,i)}
	  (\pi_\ast [\partial_i^+e]-\pi_\ast [\partial_i^-e]).
$$

\begin{example}\label{exm:1}
    Let $K$ be a simplicial complex on $\{ 1, 2, 3\}$ with a shelling $\{1,2\}$, $\{2,3\}$ and $\{1,3\}$.
    Their restrictions are $\emptyset$, $\{3\}$, and $\{1,3\}$,  respectively.
    We consider a characteristic function $\lambda$ represented by a matrix
$
  \begin{pmatrix}
	1 & 0 & 1 \\
	0 & 1 & 1
  \end{pmatrix}
$.
Then, it satisfies the non-singularity condition, and its kernel is generated by $(1,1,1) \in \Z_2^3$.
    As
\begin{align*}
    \RZ_K &= D^1 \times D^1 \times S^0 \cup S^0 \times D^1 \times D^1 \cup D^1 \times S^0 \times D^1 \\
& = \partial (D^1 \times D^1 \times D^1) \cong S^2,
\end{align*}
and $\RZ_K \to Y$ is a double cover of $Y = M^\R(K,\lambda)$, one can see that $Y$ is the real projective plane $\RP^2$.
    We set $I_{12} = \underline{01}_1 \times \underline{01}_2 \times 1_3$, $I_{23} = 1_1 \times \underline{01}_2 \times \underline{01}_3$, and  $I_{13} = \underline{01}_1 \times 1_2 \times \underline{01}_3$.
    One can see that all canonical faces are contained in one of $I_{12}$, $I_{23}$, and $I_{13}$.
    All faces of $I_{12}$ are canonical. Among the faces of $I_{23}$, $\partial_3^-I_{23}$ is not canonical.
    We take $g_{\partial_3^-I_{23}}$ as in Lemma~\ref{lem:g}, then $g_{\partial_3^-I_{23}} = (1,1,1) \in \ker \lambda$ sends $\pi\partial_3^-I_{23}$ to
 $\pi\partial_1^-I_{12}$ with the orientation reversed.
 Likewise,
 $1$-cells $[\pi\partial_3^-I_{13}]$
 and $[\pi\partial_1^-I_{13}]$ are identified with the image of
 the canonical ones,
 $-[\pi\partial_2^-I_{12}]$ and $-[\pi\partial_2^-I_{23}]$,
 respectively.
 Thus, one obtains the cell structure of $Y \cong \RP^2$, as illustrated in Figure~\ref{fig:E2}.
 \begin{figure}
   \begin{center}
\begin{tikzpicture}[scale=0.8]
\foreach \x in {5}
{
    \filldraw[fill=gray!50!white] (0,2) -- ++(0.8, 0.8) -- ++(2,0) -- ++(-0.8,-0.8) -- cycle;
    \draw [->] (1, 2.4) arc (-180:90:11pt and 7pt);
    \filldraw[thick, fill=gray!50!white] (0,2+3) -- ++(0.8, 0.8) -- ++(2,0) -- ++(-0.8,-0.8) -- cycle;
    \filldraw[thick, fill=gray!50!white] (0,2+3+2) -- ++(0.8, 0.8) -- ++(2,0) -- ++(-0.8,-0.8) -- cycle;
    \draw (1.4, 7.4) node {$I_{12}$};
    \draw [->] (1.4, 4.3)-- ++(0,-0.8);
    \draw (1.4, 3.9) node[right] {$\pi$};

    \filldraw[fill=gray!50!white] (\x,0) -- ++(2, 0) -- ++(0,2) -- ++(-2,0) -- cycle;
    \draw [->] (\x+0.5, 1) arc (-180:90:14pt);
    \filldraw[fill=gray!50!white] (\x,2) -- ++(0.8, 0.8) -- ++(2,0) -- ++(-0.8,-0.8) -- cycle;
    \draw [->] (\x+1, 2.4) arc (-180:90:11pt and 7pt);
    \draw [->, ultra thick] (\x,0)-- ++(1.2,0); \draw [ultra thick] (\x,0)-- ++(2,0);
    \draw [->, ultra thick] (\x+2.8,2.8)-- ++(-1.2,0); \draw [ultra thick] (\x+0.8,2.8)-- ++(2,0);
    \filldraw[thick, fill=gray!50!white] (\x+0.8,2+3+0.8) -- ++(2, 0) -- ++(0,2) -- ++(-2,0) -- cycle;
    \filldraw[thick, fill=gray!50!white] (\x,2+3) -- ++(2, 0) -- ++(0,2) -- ++(-2,0) -- cycle;
    \draw (\x+1, 6) node {$I_{23}$};
    \draw [->] (\x+1.4, 4.3)-- ++(0,-0.8);
    \draw (\x+1.4, 3.9) node[right] {$\pi$};

    \filldraw[fill=gray!50!white] (\x*2,0) -- ++(2, 0) -- ++(0,2) -- ++(-2,0) -- cycle;
    \draw [->] (\x*2+0.5, 1) arc (-180:90:14pt);
    \filldraw[fill=gray!50!white] (\x*2,2) -- ++(0.8, 0.8) -- ++(2,0) -- ++(-0.8,-0.8) -- cycle;
    \draw [->] (\x*2+1, 2.4) arc (-180:90:11pt and 7pt);
    \draw [->, ultra thick] (\x*2,0)-- ++(1.2,0); \draw [ultra thick] (\x*2,0)-- ++(2,0);
    \draw [->, ultra thick] (\x*2+2.8,2.8)-- ++(-1.2,0); \draw [ultra thick] (\x*2+0.8,2.8)-- ++(2,0);
    \filldraw[fill=gray!50!white] (\x*2+2,0) -- ++(0.8, 0.8) -- ++(0,2) -- ++(-0.8,-0.8) -- cycle;
    \draw [->] (\x*2+2.15, 1.4) arc (-180:90:7pt and 14pt);
    \draw [->>, ultra thick] (\x*2+2,0)-- ++(0.6,0.6); \draw [ultra thick] (\x*2+2,0)-- ++(0.8,0.8);
    \draw [->>, ultra thick] (\x*2+0.8,2.8)-- ++(-0.6,-0.6); \draw [ultra thick] (\x*2+0.8,2.8)-- ++(-0.8,-0.8);
    \draw [->>>, ultra thick] (\x*2+2.8,0.8)-- ++(0,1.2); \draw [ultra thick] (\x*2+2.8,0.8)-- ++(0,2);
    \draw [->>>, ultra thick] (\x*2,2)-- ++(0,-1.2); \draw [ultra thick] (\x*2,2)-- ++(0,-2);
    \filldraw[thick, fill=gray!50!white] (\x*2+2,5) -- ++(0.8, 0.8) -- ++(0,2) -- ++(-0.8,-0.8) -- cycle;
    \draw (\x*2+2.4, 6.2) node {$I_{13}$};
    \filldraw[thick, fill=gray!50!white] (\x*2,2+3) -- ++(0.8, 0.8) -- ++(0,2) -- ++(-0.8,-0.8) -- cycle;
    \draw [->] (\x*2+1.4, 4.3)-- ++(0,-0.8);
    \draw (\x*2+1.4, 3.9) node[right] {$\pi$};
}
\end{tikzpicture}
   \end{center}
   \caption{The decomposition of $Y \cong \RP^2$ in Example~\ref{exm:1}}
    \label{fig:E2}
 \end{figure}

\end{example}

\subsection{The transfer and the cochain complex of $Y$}
  Let $\underline{01}^\ast $, $1^\ast $, and $0^\ast $ be the
  oriented dual of $[\underline{01}]$, $[1]$, and $[0]$,
  respectively. Let $(C^\ast (\RZ_K),d)$ be the cochain complex
  dual to $(C_\ast (\RZ_K),\partial)$ (with respect to the
  $\Hom$-functor).
  We write an oriented dual cell $e^\ast $
  in the form $e^\ast =e_1^\ast \otimes \cdots \otimes e_m^\ast$,
  if $[e]=[e_1] \otimes \cdots \otimes [e_m]$.
  For disjoint subsets $\sigma$ and $\tau$ of $[m]$,
  let $u_{\sigma}t_{\tau}\in C^\ast (\RZ_K)$ be the cochain
  \begin{equation}	\label{def:ut}
	u_{\sigma}t_{\tau}=\alpha_1 \otimes \cdots \otimes \alpha_m,  \quad
	\alpha_i=\begin{cases}
	  \underline{01}_i^\ast     & i\in\sigma\\
	  1_i^\ast   & i\in\tau\\
	  1_i^\ast +0_i^\ast   &\text{otherwise}.
	\end{cases}
  \end{equation}
  When $\sigma=\tau=\emptyset$, we denote the cochain by
  the void word $\oslash$.


For $g=(g_1, \ldots, g_m) \in\Z_2^m$, by dualizing \eqref{eq:ref}, we have the cochain map
 $g^\ast \colon C^\ast (\RZ_K)\to C^\ast (\RZ_K)$ such that
 \begin{equation}    \label{eq:ref2}
   g^\ast (e^\ast )=g_1^\ast (e_1^\ast) \otimes \cdots \otimes g_m^\ast (e_m^\ast)=(-1)^{(\sigma_e,g)}(g \cdot e)^\ast ,
 \end{equation}
 where $g_i^\ast (\underline{01}_i^\ast )=-\underline{01}_i^\ast $, $g_i^\ast (1_i^\ast )=0_i^\ast $, and
 $g_i^\ast (0_i^\ast )=1_i^\ast $ if $g_i=1$.

 \begin{lemma}[{\cite[Section 3.2]{Cai2017}}]\label{lem:basis}
As an abelian group, $(C^\ast (\RZ_K),d)$ is generated
by
$\{u_{\sigma}t_{\tau}\}_{\sigma,\tau}$,
with $\sigma,\tau$ running through disjoint subsets of $[m]$
such that $\sigma\in K$.
The coboundary operator follows
\begin{equation}
   d (u_{\sigma}t_{\tau})=
  \sum_{i\in\tau\atop\sigma\cup\{i\}\in K}
  (-1)^{(\sigma,i)}u_{\sigma\cup\{i\}}t_{\tau\setminus\{i\}},
  \label{def:d}
\end{equation}
  where $(\sigma,i)=\card (\{j\in\sigma\mid j<i\})$.
\end{lemma}

\begin{proof}
  For the first statement, it suffices
  to show that each dual cell
  $e^\ast $ can be expressed as a $\Z$-linear
  sum by cochains of the form $u_{\sigma}t_{\tau}$.

  We use an induction on $k=\card (\tau_e^-)$.
  For the case when $k=0$, as $\sigma_e\cup\tau_e^+=[m]$, we have $e^\ast =u_{\sigma_e}t_{\tau_e^+}$.
  Assume that this is true for all $k<\ell$,
  and suppose $\card (\tau_e^-)=\ell$.
  Expanding \eqref{def:ut} as a sum of dual cells,
  we see that each summand in
  $u_{\sigma_e}t_{\tau_e^+}-e^\ast $ is in
  the form $(e'')^\ast $ so that $\card (\tau^{-}_{e''})<\ell$.
  By assumption, $u_{\sigma_e}t_{\tau_e^+}-e^\ast $ can
  be expressed as a $\Z$-linear sum as desired,
  so can $e^\ast $.
  Therefore, the first statement holds by induction.

  By dualizing \eqref{eq:bd}, we have the evaluation
  \begin{equation}
	\langle de^\ast , [e']\rangle=
  \langle e^\ast ,\partial [e']\rangle=\begin{cases}
	(-1)^{(\sigma,i)} & \text{if $\partial_i^+e'=e$,}\\
	-(-1)^{(\sigma,i)} & \text{if $\partial_i^-e'=e$,}\\
	0                  & \text{otherwise}.
  \end{cases}
	\label{eq:pf1}
  \end{equation}
 Note that
 if the evaluation is non-zero, then the label $i$ above is
 uniquely determined by $e$ and $e'$.
 If $i\in\sigma$, then $\langle d(u_{\sigma}t_{\tau}), [e']\rangle$ vanishes.
 If $i\in[m]\setminus (\sigma\cup\tau)$, then the pair
 $(\partial_i^+e')^\ast $ and
 $(\partial_i^-e')^\ast $ both appear as summands
 in $u_{\sigma}t_{\tau}$, by \eqref{eq:pf1},
 they cancel each other out in the evaluation.
 Therefore, only for $i\in\tau$, the dual cell $(\partial_i^+ e')^\ast $
 is a summand in  $u_{\sigma}t_{\tau}$, and
 \[\langle d(u_{\sigma}t_{\tau}), [e']\rangle=(-1)^{(\sigma,i)}.\]
 As both sides of \eqref{def:d} give the same evaluation
 on each cell, they coincide.
\end{proof}

The characteristic function $\lambda$ can be expressed by an $n \times m$ $\Z_2$-matrix $\Lambda$. Denote by $\row \lambda$ the row space of $\Lambda$ in $\Z_2^m$.

\begin{lemma}\label{lem:int}
  For each $\sigma\in K$, there exists a unique element
  $\omega\in \row \lambda$ such that
  \begin{equation}
	\omega \cap f(\sigma)=\sigma,
	\label{eq:cap}
  \end{equation}
  where $f(\sigma)$ is the first facet in the shelling
  that contains $\sigma$.
\end{lemma}

\begin{proof}
  Without loss of generality, suppose $f(\sigma)=\{1,2,\ldots,n\}$
  and $\lambda|_{f(\sigma)}\colon\Z_2^n\to\Z_2^n$ is the identity.
  Suppose $(\lambda_1,\ldots,\lambda_n)$ are the row vectors of
  $\Lambda$.
  Set $\omega=\sum_{i\in\sigma}\lambda_i$.
  Then $\omega$ satisfies \eqref{eq:cap}, which is unique.
\end{proof}

Recall that the \emph{transfer homomorphism}
$T_\ast \colon C_\ast (Y)\to C_\ast (\RZ_K)$ is a chain map
defined by
$$
   T_\ast (\pi_\ast ([e]))=\sum_{g'\in\ker \lambda}g'_\ast ([e]),
$$
where $e$ is a cubical cell of $\RZ_K$.
It can be checked that the
definition is independent of the choice of the
pre-image $[e]$, and $T_\ast$ is invariant in the sense that
\begin{equation}
  g_\ast \circ T_\ast =T_\ast ,
  \label{eq:inv}
\end{equation}
where
$g_\ast \colon C_\ast (\RZ_K)\to C_\ast (\RZ_K)$ is induced by $g\in\ker \lambda$.
Let $T^\ast \colon C^\ast (\RZ_K)\to C^\ast (Y)$ be the dual of $T_\ast $.
By Lemma~\ref{lem:basis}, the dual chain complex
$(C^\ast (Y),d)$ of $(C_\ast (Y),\partial)$ is generated by
$\{T^\ast (u_{\sigma}t_{\tau})\}_{\sigma,\tau}$.

Recall that a cubical cell
$e$ is canonical, if $\tau_e^-\subset f(\sigma_e)$.
By Lemma~\ref{lem:g},
$C^\ast (Y)$ has a basis $\{T^\ast (e^\ast )\}_e$
with $e$ running through canonical cells of
$\RZ_K$. We would like to express
each cochain $T^\ast (u_{\sigma}t_{\tau})$ as a $\Z$-linear sum of these
basis elements.

A cochain $c$ is said to be \emph{divisible} by an integer $r$,
if there exists a cochain $c'$ so that $c=rc'$, and is said to be \emph{primitive} if it is not divisible by any integer greater than $1$. For an element $\omega \in \Z_2^m$, $K_\omega$ is the full subcomplex of $K$ induced by the subset of $[m]$ associated to $\omega$.

\begin{proposition}\label{prop:main}
 Given $\omega\in \row \lambda$
 and $\sigma\in K_{\omega}$ with
 $\card (\sigma)=k$,  the cochain
 $T^\ast (u_{\sigma}t_{\omega\setminus\sigma})\in C^k(Y)$
 is divisible by $2^{\mu_k(\omega)}$,
 where
$$
	 \mu_k(\omega)=\max\{m-n+k-\card (\omega ),0\}.
$$
   In particular, if $\omega \cap f(\sigma)=\sigma$, then
   $m-n+k-\card (\omega )=\mu_k(\omega)\geq 0$, and
  \begin{equation}
   T^\ast (u_{\sigma}t_{\omega\setminus\sigma})=
   2^{\mu_k(\omega)}T^\ast (u_{\sigma}t_{[m]\setminus f(\sigma)}),
   \label{eq:spe}
 \end{equation}
where $T^\ast (u_{\sigma}t_{[m]\setminus f(\sigma)})$ is primitive in
$C^\ast (Y)$.
 \end{proposition}

\begin{proof}
  Let $\tau_1=(\omega  \cap f(\sigma))\setminus \sigma $
  and $\tau_2=[m]\setminus (f(\sigma)\cup \omega )$.
  We see that $\omega  \cap f(\sigma)=\tau_1\cup\sigma$,
  and $\omega  \setminus f(\sigma)=[m]\setminus (f(\sigma)\cup\tau_2)$.
  As $\tau_1\cap\sigma=\emptyset$ and $f(\sigma)\cap\tau_2=\emptyset$,
  \begin{align*}
	\card (\omega)&=
	\card (\omega\cap f(\sigma))+
	\card (\omega\setminus f(\sigma))\\
	&=\card (\tau_1)+k+m-n-\card (\tau_2),
  \end{align*}
  namely
	\[ \card (\tau_2)-\card (\tau_1)
	  =m-n+k-\card (\omega)=\mu_k(\omega).\]
	Denote $k_1=\card (\tau_1)$ and $k_2=\card (\tau_2)$.
	For $\ell\in\tau_2$, let $g_\ell\in\ker \lambda$ be
	the unique element so that $\sigma_{g_\ell}\setminus f(\sigma)=\{\ell\}$,
	where $\sigma_{g_\ell}\subset [m]$ is the set of non-zero
	entries of $g_\ell$ (see Lemma~\ref{lem:g}). Let $G_{\tau_2}$ be
	the subgroup generated by $\{g_\ell\mid \ell\in\tau_2\}$, the
	order of which is $2^{k_2}$.

	For $S\subset [m]$, define $(S,g)=\card (S\cap\sigma_g)$.
	We see that $g\in\ker \lambda$, $\omega\in \row \lambda$ implies
	that $(\omega,g)$ is an even number.
	For $g\in G_{\tau_2}$, $(\omega\setminus f(\sigma),g)=0$, thus
	\begin{equation}
	 0=(\omega,g)=(\omega\cap f(\sigma),g)+(\omega\setminus f(\sigma),g)=
	 (\sigma,g)+(\tau_1,g) \quad \mathrm{mod} \ 2.
	  \label{eq:key1}
	\end{equation}
	We expand $u_{\sigma}t_{\omega\setminus\sigma}$
	with respect to $\tau_2$, as a sum
	of $2^{k_2}$ terms. By \eqref{def:ut},
	\[u_{\sigma}t_{\omega\setminus\sigma}=\sum_{g\in G_{\tau_2}}\alpha_g,
	  \quad \alpha_g=\otimes_{i=1}^m\alpha_i, \quad \alpha_i=\begin{cases}
		\underline{01}_i^\ast   & i\in\sigma \\
		0_i^\ast   & i\in\sigma_g\cap\tau_2 \\
		1_i^\ast +0_i^\ast    & i\in f(\sigma)\setminus\omega\\
		1_i^\ast     & \text{otherwise}.
	  \end{cases}\]
	  Let $p\colon G_{\tau_2}\to \Z_2^{k_1}$
	  be the projection sending $g=(g_i)_{i=1}^m$ to
	  $p(g)=(g_i)_{i\in\tau_1}$.
	  For $g\in G_{\tau_2}$, $g^\ast $ turns the $0_i^\ast $-components in $\tau_2$
	  of $\alpha_g$ into $1_i^\ast $-components, hence
	  by \eqref{eq:ref2}, we have
	  \begin{align*}
		g^\ast (\alpha_g)=(-1)^{(\sigma,g)}\alpha'_g
		\quad
		\alpha'_g=\otimes_{i=1}^m\alpha_i', \quad \alpha_i'=\begin{cases}
		\underline{01}_i^\ast   & i\in\sigma \\
		0_i^\ast   &  i\in\sigma_{p(g)} \\
		1_i^\ast +0_i^\ast    & i\in f(\sigma)\setminus\omega\\
		1_i^\ast     & \text{otherwise}.
	  \end{cases}
  \end{align*}
  Note that by \eqref{eq:key1}, $(-1)^{(\sigma,g)}\alpha'_g=
  (-1)^{\card (\sigma_{p(g)})}\alpha'_g$,
  which means that $g^\ast (\alpha_g)$ is only determined by the image
  $p(g)$,
  thus we denote
  \[\beta_{p(g)}=g^\ast (\alpha_g)=(-1)^{\card (\sigma_{p(g)})}\alpha'_g.\]
  Together with \eqref{eq:inv}, we see that
  \begin{align*}
	T^\ast (u_{\sigma}\tau_{\omega\setminus\sigma})
	&=\sum_{g\in G_{\tau_2}}T^\ast (\alpha_g)
	=\sum_{g\in G_{\tau_2}}T^\ast (g^\ast (\alpha_g))\\
	&=\sum_{g\in G_{\tau_2}}T^\ast (\beta_{p(g)})=N
	T^\ast (\sum_{g'\in \mathrm{Im}p}\beta_{g'}),
	\numberthis\label{eq:fin1}
  \end{align*}
	where $N$ is the order of the kernel of $p$, which is clearly
	divisible by $2^{k_2-k_1}=2^{\mu_{k}(\omega)}$ if $k_2 \geq k_1$.

	In the special case $\tau_1=\emptyset$, we have
	$\mu_k(\omega)=\card (\tau_2)$,
	$\sigma_{p(g)}=\emptyset$, and
	\[\beta_{p(g)}=\alpha'_g=u_{\sigma}t_{[m]\setminus f(\sigma)}\]
	for all $g\in G_{\tau_2}$;
	now \eqref{eq:spe} follows
	from \eqref{eq:fin1}, as $N=2^{k_2}$.
	Each dual cell appearing in the sum
	$u_{\sigma}t_{[m]\setminus f(\sigma)}$ is the dual
	of a canonical one, hence
	$T^\ast (u_{\sigma}t_{[m]\setminus f(\sigma)})$ is primitive.
\end{proof}

\section{The cellular (co)chain complexes of $K_\omega$}\label{sec:simp_cpx_of_K}

To describe the cohomology of $Y$ in terms of $K_\omega$ for $\omega \in \row \lambda$, we will construct an explicit cellular (co)chain complex of $K_\omega$ in this section.

With respect to $K\subset 2^{[m]}$, we denote by $(\wt{C}_\ast (K),\partial')$
the \emph{augmented (ordered) simplicial chain complex} of $K$:
$\wt{C}_\ast (K)=\bigoplus_{q\geq-1} C_p(K)$,
where $C_q(K)$ is generated by oriented simplices
$[\sigma]=[i_0,\ldots,i_q]$ ($\sigma\in K$),
$i_0<\cdots<i_q$ (formally $C_{-1}(K)$ is generated
by $[\emptyset]$),
together with the boundary operator
\begin{equation}
   \partial'[\sigma]=\begin{cases}
	\sum_{k=0}^q(-1)^{k}
	[\sigma\setminus\{i_k\}] & q>0\\
	[\emptyset] & q=0\\
	0           & q=-1.
  \end{cases}
  \label{eq:partial'}
\end{equation}
Here, the orientation follows the rule that a permutation of $i_k,i_j\in\sigma$ gives a $(-1)$-sign.
It can be shown that $(\wt{C}_\ast (K),\partial')$
is chain-homotopy equivalent to the
reduced singular chain complex of
$K$ (see \cite{Munkres1984book} for details).

\begin{definition}
  Let $K$ be a simplicial complex on $[m]$.
  For $\sigma \subset [m]$ such that $\sigma \not\in K$, the simplicial complex
  $K \cup 2^\sigma$
  is called the \emph{regular (simplicial) expansion} of $K$ along $\sigma$
if $\dim \sigma \geq 0$ and
  the set $\{\tau\subset\sigma\mid\tau\not\in K\}$
  has a unique, minimal element, which is denoted by
  $r(\sigma)$.
  We still call $r(\sigma)$ the \emph{restriction} of $\sigma$.
If $\sigma$ is a singleton, then $r(\sigma)=\sigma$. Note that $r(\sigma)\not=\emptyset$ as
  $\emptyset\in K$.
\end{definition}

Geometrically, $K'$ is a regular expansion of $K$
if and only if the intersection of geometrical realizations of $\sigma$ and $K$ is
a union of codimension-one faces of
$\sigma$.
It is easy to see that, there is a
strong deformation retraction $K'\to K$
along $\sigma$ unless
$r(\sigma)=\sigma$, i.e.,
the whole boundary of $\sigma$ is
contained in $K$.

\begin{lemma}\label{lem:rho1}
 Suppose $K'$ is a regular expansion of $K$ along
 $\sigma$ with $r(\sigma)\not=\sigma$.
 Then there is a simplicial chain map
 $\rho\colon \wt{C}_\ast (K')\to \wt{C}_\ast (K)$
 whose restriction to $C_\ast (K)$ is the identity.
\end{lemma}

\begin{proof}
 We can choose a simplicial approximation of the retraction
 $ K' \to K$, which is the identity on
 $K$.
 This can be done since $\sigma$ is a simplex.
 To preserve the degree,
  if in the approximation a simplex is mapped (geometrically) to
 another one whose dimension is strictly lower, then
 we set its image to be zero in the target chain group. It
 can be checked that in this way we
 obtain a chain map (cf.\
 \cite[pp.~62--63]{Munkres1984book}).
 Then the induced chain map
 is as desired.
\end{proof}

\begin{definition}
  We say that $\sigma_1,\sigma_2,\ldots,\sigma_\ell$ is a
  \emph{regular expanding sequence} of $K$, if
   $K$ admits a filtration
   $2^{\sigma_1}=K_1\subset K_2\subset\cdots\subset K_\ell=K$
   such that $K_{j+1}$ is
  a regular expansion of $K_j$ along $\sigma_{j+1}$ for all $j=1,\ldots,s-1$,
  where $2^{\sigma_1}$ denotes all subsets of $\sigma_1\not=\emptyset$.
  In the sequence above, a simplex $\sigma_{i}$
  is called \emph{critical}
  if $r(\sigma_j)=\sigma_j$.
  We shall denote by $\Cri(K)$ the set of critical faces
  in the regular expanding sequence.
  Formally, we set $r(\sigma_1)=\emptyset$.
\end{definition}

We will see that, up to homotopy, $K$
has a cell decomposition by critical
simplices. If $\Cri (K)=\emptyset$, $K$ deformation
retracts onto a vertex.
Here is the key lemma for our main theorem.

\begin{lemma}\label{prop:key2}
  Suppose $\sigma_1,\ldots,\sigma_\ell$ is a regular expanding sequence  of $K$
  with the filtration $K_1\subset\cdots\subset K_\ell$.
  Let $\ol{C}_\ast (K)$ be the free abelian group generated
  by $\Cri (K)$, with each critical simplex suitably
  oriented (which is trivial if $\Cri (K)=\emptyset$), and
  let $(\wt{C}_\ast (K),\partial')$ be the simplicial chain complex.
  Then $\ol{C}_\ast (K)$ admits a boundary operator $\ol{\partial}'$ to
  be a chain complex,
  together with a (graded) chain map
  $\rho\colon \wt{C}_\ast (K)\to \ol{C}_\ast (K)$, with the following properties:
  \begin{enumerate}
	\item [(a)] $\rho$ is an identity when restricted to $\Cri (K)$,
	  \[\rho([\sigma])=[\sigma],\quad \sigma\in \Cri (K);\]
	\item [(b)] $\rho$ induces an isomorphism in homology;
	\item [(c)] let
	  $\rho^\ast $ be the dual of $\rho$ (by the $\Hom $-functor),
	  if $\sigma_{j+1}$ is critical, then
	  \begin{equation}
		\rho^\ast (\sigma_{j+1}^\ast )=\sigma_{j+1}^\ast +c_{j+1}^+, \label{eq:>j+1}
	  \end{equation}
	  where $\sigma_{j+1}^\ast $ is the oriented dual of $[\sigma_{j+1}]$, and
	  $c_{j+1}^+$ is a $\Z$-linear sum with dual simplices of
	  the form $\sigma_t^\ast $ involved for $t>j+1$.
  \end{enumerate}
\end{lemma}

\begin{proof}
    We construct $\rho$ inductively from
	$K_j$ to $K_{j+1}$, $j=1,\ldots,l-1$. First consider $K_1=2^{\sigma_1}$.
  As $\Cri (K_1)$ is empty, by definition, we
  define $\rho([\sigma])=0$ for all $\sigma\subset\sigma_1$ (including
  the case $\sigma=\emptyset$). Suppose the chain map
  $\rho_j\colon \wt{C}_\ast (K_j)\to \ol{C}_\ast (K_j)$
  has been defined as desired.
  \begin{enumerate}
	\item [(I)] If $\sigma_{j+1}$ is not critical, let
	  $f_{j+1}\colon \wt{C}_\ast (K_{j+1})\to \wt{C}_\ast (K_j)$ be the map given in Lemma~\ref{lem:rho1},
  and define $\rho_{j+1}$ as the composition $\rho_j\circ f_{j+1}$;
\item [(II)] Otherwise, if $\sigma_{j+1}$ is critical, we define
  \[\ol{\partial}'=\rho_j(\partial'[\sigma_{j+1}])\]
  and $\rho_{j+1}([\sigma_{j+1}])=[\sigma_{j+1}]$, preserving the orientation:
  $\rho_{j+1}$ coincides with
  $\rho_j$ on $\wt{C}_\ast (K_j)$.
  \end{enumerate}
   By induction, it can be checked that
   the chain map $\rho\colon \wt{C}_\ast (K)\to \ol{C}_\ast (K)$ is well-defined,
   which satisfies both (a) and (b). Property (c) follows from (I).
\end{proof}

Note that in a regular expanding sequence of $K$,
compared with a shelling,
each $\sigma_j$ is not necessarily a facet and $K$ is not necessarily
pure. We have the following relation between them.

\begin{proposition}\label{prop:esKw}
  Suppose $\sigma_1,\ldots,\sigma_s$ is a shelling of $K\subset 2^{[m]}$.
  Given $\omega\subset [m]$, let
  $\sigma_{j_1}\cap\omega, \ldots, \sigma_{j_\ell}\cap\omega$ be
  the (non-repeating)
  list of non-empty simplices in $\{\sigma_j\cap\omega\}_{j=1}^s$.
  Then $\sigma_{j_1}\cap\omega,\ldots,\sigma_{j_\ell}\cap\omega$
  is a regular expanding sequence
  of the full subcomplex $K_{\omega}$, with their restrictions
  $r(\sigma_{j_1}),\ldots,r(\sigma_{j_\ell})$, respectively,
  where $r(\sigma_{j_k})$
  is the restriction of $\sigma_{j_k}$ in the given shelling for $k=1,\ldots,\ell$.
\end{proposition}
\begin{proof}
  Recall that, by definition,
  the restriction $r(\sigma_{j+1})$ of $\sigma_{j+1}$ in the
  shelling is the minimal element in the set
  $\{\sigma\subset \sigma_{j+1}\mid\sigma\not\in K_{j}\}$.
  Suppose the intersection $\sigma_{j+1}\cap\omega$,
  as a simplex in the full subcomplex $K_{j+1}|_{\omega}$ of $K_{j+1}$,
  is not empty.
  If $r(\sigma_{j+1})\subset \omega$, then,
  by definition, $K_{j+1}|_{\omega}$ is a regular expansion of $K_j|_{\omega}$
  along $\sigma_{j+1}\cap\omega$, whose restriction is again $r(\sigma_{j+1})$;
  otherwise, $r(\sigma_{j+1})\cap\omega$ is a proper subset of $r(\sigma_{j+1})$,
  which means that any simplex $\sigma\in K_{j+1}|_{\omega}$ cannot contain
  $r(\sigma_{j+1})$, thus, by minimality, $\sigma\in K_j$, hence $\sigma\in K_j|_{\omega}$:
  $K_{j+1}|_{\omega}=K_j|_{\omega}$.
  The statement follows by induction.
\end{proof}

Let $(\ol{C}^\ast (K),\ol{d}')$ and $(\wt{C}^\ast (K),d')$ be the dual of the
chain complex of $(\ol{C}_\ast ,\ol{\partial}')$ and $(\wt{C}_\ast (K),\partial')$, respectively, and let
$\rho^\ast \colon\ol{C}^\ast (K)\to \wt{C}^\ast (K)$ be the
dual of $\rho$.

Let $\sigma^\ast \in \wt{C}^\ast (K)$ be the oriented dual of
$[\sigma]\in \wt{C}_\ast (K)$
(as $K$ is finite, $C^\ast (K)$ is
generated by dual simplices).
It can be checked that, by dualizing \eqref{eq:partial'}, we have
\begin{equation}
   d'\sigma^\ast =\sum_{(\sigma\cup\{i\})\in K}
(-1)^{(\sigma,i)}(\sigma\cup\{i\})^\ast \quad
(d'\emptyset^\ast =\sum_{\{i\}\in K}i^\ast ),
  \label{eq:d'}
\end{equation}
where $(\sigma,i)=\card (\{j\in\sigma\mid j<i\})$.

Fix $\omega\subset [m]$ and
consider the map
\begin{equation}
  \varphi\colon \wt{C}^\ast (K_{\omega})\to C^{\ast +1}(\RZ_K)
  \label{def:vphi}
\end{equation}
that sends $\sigma^\ast $ to
$u_{\sigma}t_{\omega\setminus\sigma}$.
A comparison of \eqref{eq:d'} and \eqref{def:d} shows that
$\varphi$ is a (degree-increasing) cochain
map with respect to differentials
$d'$ and $d$.

Now suppose $\omega\in \row \lambda$ and let $\sigma_1,\ldots,\sigma_s$ be a shelling of $K$ with
restrictions $r(\sigma_1),\ldots,r(\sigma_s)$. Then, by Proposition~\ref{prop:esKw}, we have a regular expanding sequence of $K_\omega$ as
$$
\sigma_{j_1}\cap\omega, \ldots,\sigma_{j_\ell}\cap\omega.
$$
The following lemma is straightforward.

\begin{lemma}\label{lem:int2}
  The simplex $\sigma_{j_k}\cap\omega$, $k=1,\ldots,l$,
  is critical if and only if
  $\sigma_{j_k}\cap\omega=r(\sigma_{j_k})$.
   That is,
  the cochain complex $\ol{C}^\ast (K_{\omega},\ol{d}')$ is
  generated by $\{\sigma_j\mid\sigma_j\cap\omega=r(\sigma_j)\}$.
\end{lemma}

\section{Main Theorem and proof of Theorem~\ref{thm:integral_cohom_of_Y}}\label{sec:5}
Let $K$ be an $(n-1)$-dimensional pure simplicial complex on $[m]$ with a shelling $\sigma_1,\ldots, \sigma_s$, and $r(\sigma_j)$ the restriction of $\sigma_j$ in the shelling for $j=1, \ldots, s$.
Let $\lambda \colon \Z_2^m \to \Z_2^n$ be a characteristic function satisfying the non-singularity condition.
We write $Y = M^\R(K,\lambda)$ for the real toric space associated to $K$ and $\lambda$.
%

For each non-negative integer $k \geq 0$, let $\phi_k|_{\omega}$ be the map given by
\[ \phi_k|_{\omega}=\frac{1}{2^{\mu_k(\omega)}}T^\ast \circ\varphi\circ\rho^\ast  \colon
  \ol{C}^{k-1}(K_{\omega}) \to C^k(Y),\]
 where $T^\ast $ is the dual transfer homomorphism and $\mu_k(\omega)=m-n+k-\card (\omega)$.
The coefficient $1/{2^{\mu_k(\omega)}}$
 is multiplied to make the image  primitive (see Proposition~\ref{prop:main}).

 \begin{lemma} \label{lem:phi_is_cochain_map}
   Let $\phi|_{\omega}\colon\ol{C}^\ast (K_{\omega}) \to C^{\ast +1}(Y)$ be
   the homomorphism whose restriction to $\ol{C}^{k-1}(K_{\omega})$ is
 $\phi_k|_{\omega}$ for $k=1,\ldots,n$.
 If we endow $\ol{C}^\ast (K_{\omega})$
 with the coboundary operator $2\ol{d}'$ rather than $\ol{d}'$, then it is a cochain map
 \[(\ol{C}^\ast (K_{\omega}),2\ol{d}') \to (C^{\ast +1}(Y),d),\]
where the coboundary operator
	  $2\ol{d}'$ means $(2\ol{d}')c=2(\ol{d}'c)$ for $c\in \ol{C}^\ast (K_\omega)$.
 \end{lemma}

 \begin{proof}
   As a composition of cochain maps, $T^\ast \circ\varphi\circ\rho^\ast $ is a cochain map
   with respect to $\ol{d}'$. Choose $c\in \ol{C}^{k-1}(K_{\omega})$, then we
  can see that
   \begin{align*}
	 \phi(2\ol{d}'c)&=\frac{1}{2^{\mu_{k+1}(\omega)}}T^\ast \circ\varphi\circ\rho^\ast (2\ol{d}'c)\\
	 &=\frac{2}{2^{\mu_{k+1}(\omega)}}d\circ T^\ast \circ\varphi\circ\rho^\ast (c)\\
	 &=d\left(\frac{1}{2^{\mu_k(\omega)}}
	 T^\ast \circ\varphi\circ\rho^\ast (c)\right)
	 =d\phi(c).
   \end{align*}
 \end{proof}

We put
$$\ol{C}^{\ast }_{\lambda}=
 \bigoplus_{\omega\in \row \lambda}\ol{C}^{\ast }(K_{\omega}),
$$
 and let
 $\phi\colon (\ol{C}^{\ast }_{\lambda},2\ol{d}')\to (C^{\ast +1}(Y),d)$
 be the map whose restriction to $\ol{C}^{\ast }(K_{\omega})$ is $\phi|_{\omega}$.
By Lemma~\ref{lem:phi_is_cochain_map}, $\phi$ is indeed a cochain map increasing the degrees by one.
The following theorem is the key theorem of this paper.

\begin{theorem}\label{thm:main}
The cochain map $\phi\colon (\ol{C}^\ast _{\lambda},2\ol{d}')	  \to (C^{*+1}(Y),d)$ satisfies the following properties.
	  \begin{enumerate}
		\item [(a)]
		The cochain complex $\ol{C}_{\lambda}^\ast $ is generated by
$\{r(\sigma_{j})\in K_{\omega_j}\}_{j=1}^s$,
where $\omega_j\in \row \lambda$ is the unique element
that satisfies $\omega_j\cap\sigma_j=r(\sigma_i)$;
we have
 \begin{equation}
   [\phi(r(\sigma_j))]=[x_{r(\sigma_j)}]
   \label{eq:d1}
 \end{equation}
 in $H^\ast (Y;\Z_2)$, where $[\phi(r(\sigma_j))]$ denotes
 the $\mathrm{mod}$ $2$ reduction of $\phi(r(\sigma_j))$.
		\item [(b)] The map $\phi$ induces an isomorphism of
		  cohomology groups.
	  \end{enumerate}
	\end{theorem}

\begin{proof}
 The first statement in (a) follows from Lemmas \ref{lem:int} and \ref{lem:int2}.
 It remains to prove (b) and equation \eqref{eq:d1}.
Suppose
 $K_1\subset \cdots\subset K_s$ is the filtration
 associated to the shelling $\sigma_1,\ldots,\sigma_s$ of $K$,
 and let $Y_1\subset\cdots\subset Y_s$ be the filtration
 of $Y$ given in Proposition~\ref{prop:dec}.
 The proof uses an induction.
 Suppose $\omega_j\cap \sigma_j=r(\sigma_j)$, and
 denote by $\ol{C}^\ast _{\lambda}(K_j)=\bigoplus_{\omega\in \row \lambda}\ol{C}^\ast (K_j|_{\omega})$
 the cochain complex associated to $K_j$,
where  $K_j|_{\omega}$ is the full subcomplex of $K_j$ induced by $\omega$.

 First consider $j=1$.
 By definition, $r(\sigma_1)=\emptyset$, thus
 $\omega_1=\emptyset$, and $\ol{C}^\ast _{\lambda}(K_1)=\wt{C}^\ast (K_{\omega_1})$
 concentrates in degree $-1$
 and is generated by $\emptyset^\ast $; by \eqref{def:vphi}, $\varphi$ maps $\emptyset^\ast $ to
 the void word $\oslash$, and the class with representative
 $\frac{1}{2^{m-n}}T^\ast (\oslash)$
 generates $H^0(Y_1)$
 (in fact, it is the sum of all $2^n$
 dual vertices of $Y_1$, an $n$-cube). Therefore,
 (b) and \eqref{eq:d1} hold for $K_1$.

 Now suppose they hold for $K_j$. We treat $\ol{C}^\ast _{\lambda}(K_j)$ and
 $\ol{C}^\ast _{\lambda}(K_{j+1})$ as abelian groups and denote
 by $\ol{D}^\ast _{j+1}$ their difference
 $\ol{C}^\ast _{\lambda}(K_{j+1})\setminus \ol{C}^\ast _{\lambda}(K_j)$.
 We see that $\ol{D}^\ast _{j+1}$ is closed under $2\ol{d}'$, and
 denote by $(\ol{D}^\ast _{j+1},2\ol{d}')$ the corresponding
 cochain complex. Observe that the relative cochain complex
 $((\ol{C}^\ast _{\lambda}(K_{j+1}),\ol{D}^\ast _{j+1}),2\ol{d}')$ is canonically
 isomorphic to $(\ol{C}^\ast _{\lambda}(K_j),2\ol{d}')$.
 By definition, \[\ol{D}^\ast _{j+1}=
   \bigoplus_{\omega\in \row \lambda}(\ol{C}^\ast (K_{j+1}|_{\omega})
   \setminus\ol{C}^\ast (K_{i}|_{\omega}))\]
 is generated by dual simplices among
 $\{(\sigma_{j+1}\cap\omega)^\ast \}_{\omega\in \row \lambda}$ so that
 $\sigma_{j+1}\cap\omega$ is critical in $K_{j+1}|_{\omega}$ (see Lemma~\ref{prop:key2}), but by
 Lemma \ref{lem:int},
 $\omega_{j+1}$ is unique, thus
 $\ol{D}^\ast _{j+1}$ is generated by a single element
 $r(\sigma_{j+1})^\ast \in K_{j+1}|_{\omega_{j+1}}$, which is
 clearly a cocycle in dimension
 $k_{j+1}-1$, $k_{j+1}=\card (r(\sigma_{j+1}))$.

In a similar manner, consider the difference $C^\ast (Y_{j+1})\setminus C^\ast (Y_{j})$
 of dual cells, which we shall denote by $D^\ast _{j+1}$; as a subgroup
 of $(C^\ast (Y_{j+1}),d)$, $D^\ast _{j+1}$ is also closed under $d$,
 and the relative cochain complex $((C^\ast (Y_{j+1}),D^\ast _{j+1}),d)$ is canonically
 isomorphic to $(C^\ast (Y_j),d)$.
 Geometrically, up to homotopy,
 $Y_{j+1}$ is obtained from $Y_j$ by attaching
 a cell of dimension $k_{j+1}$ (see Proposition \ref{prop:dec}), and
 $D^\ast _{j+1}$ is a collection of the duals of cubical cells
 \[\{\pi (I_{\sigma})\mid
   r(\sigma_{j+1})\subset\sigma\subset\sigma_{j+1}\}\]
   where $\pi\colon\RZ_{j+1}\to Y_{j+1}$ the orbit map
   (see \eqref{eq:cube}). Thus, the cohomology of $(D^\ast _{j+1},d)$
   is infinite cyclic that concentrates in dimension
   $k_{j+1}$, and we can check directly that
   it is generated by the class with representative
   \begin{equation}
	 \frac{1}{2^{\mu_{k_{j+1}}(\omega_{j+1})}}T^\ast (u_{r(\sigma_{j+1})}
	 t_{\omega_{j+1}\setminus\sigma_{j+1}}),
	 \label{eq:image}
   \end{equation}
	 which is primitive by Proposition \ref{prop:main}
	 (by \eqref{def:d} it is a cocycle
	 as in $K_{j+1}$,
	 $\sigma_{j+1}$ is the unique facet that contains
	 $r(\sigma_{j+1})$).

We see that the cochain map $\phi$
induces a homomorphism between
the long exact sequences associated to
pairs $(\ol{C}^\ast _{\lambda}(K_{j+1}),\ol{D}^\ast _{j+1})$
and $(C^\ast (Y_{j+1}),D^\ast _{j+1})$, respectively:
$$
\xymatrix{
\cdots & \ar[l] H^t(\ol{C}^\ast _{\lambda}(K_{j+1}),\ol{D}^\ast _{j+1}) \ar[d]^{\phi}& \ar[l] H^{t}(\ol{C}^\ast _{\lambda}(K_{j+1})) \ar[d]^{\phi}& \ar[l] H^t(\ol{D}^\ast _{j+1}) \ar[d]^{\phi} & \ar[l] \cdots \\
\cdots & \ar[l] H^{t+1}(C^\ast (Y_{j+1}),D^\ast _{j+1}) & \ar[l] H^{t+1}(C^\ast (Y_{j+1})) & \ar[l] H^{t+1}(D^\ast _{j+1}) & \ar[l] \cdots
}
$$
Where, by the induction hypothesis, for any $t\geq -1$,
the first column is an isomorphism;
in  Lemma~\ref{prop:key2}(c), $c_{j+1}^+$
vanishes from \eqref{eq:>j+1} because
$r(\sigma_{j+1})^\ast =(\sigma_{j+1}\cap\omega_{j+1})^\ast $ is the last
dual simplex in $\wt{C}^\ast (K_{j+1}|_{\omega_{j+1}})$, hence
the image of $r(\sigma_{j+1})^\ast $ under $\phi$ is exactly
\eqref{eq:image}, which means that
the third column is also an isomorphism (which is non-trivial only
when $t=k_{j+1}-1$). Therefore, the middle column is an isomorphism,
from which (b) holds. By Proposition \ref{prop:DJSta},
as a vector space, $H^\ast (Y_{j+1};\Z_2)$
is obtained from $H^\ast (Y_j;\Z_2)$ by adding the basis element
$[x_{r(\sigma_{j+1})}]$, which has to be
the $\mathrm{mod}$ $2$ reduction of
\eqref{eq:image}. We see that \eqref{eq:d1} holds,
and the whole proof is completed by induction.
\end{proof}

By Lemma~\ref{prop:key2}, one can see that $(\ol{C}^\ast _{\lambda},\ol{d}')$ is
		  a direct sum $\oplus_{\omega\in \row \lambda}(\ol{C}^\ast (K_{\omega}),\ol{d}')$,
		  in which $(\ol{C}^\ast (K_{\omega}),\ol{d}')$
		  is cochain-homotopy equivalent to the
		  reduced singular cochain complex
		  of $K_{\omega}$.
We also remark that $(C^\ast (Y),d)$ is cochain-homotopy
		  equivalent to the singular cochain complex of $Y$.

    To understand the cohomology
	of $(\ol{C}^\ast _\lambda,2\ol{d}')$ and
	$(\ol{C}^\ast _{\lambda},\ol{d}')$,
	we use the normal form of a morphism
	between two finitely generated abelian groups:
    for each $i\geq 0$,
	we can find two bases for cocycles and coboundaries
	in $(\ol{C}^i_{\lambda},\ol{d}')$,
	where a coboundary is certain integral times of
	exactly one cocycle (cf.~\cite[Theorem 11.3, pp.~55--56]{Munkres1984book}).
	The same bases still work for $(\ol{C}^i,2\ol{d}')$,
	whereas those coefficients are doubled.
    Therefore, we have Theorem~\ref{thm:integral_cohom_of_Y} as a corollary of Theorem~\ref{thm:main}.


\section{The Bockstein spectral sequence and proof of Theorem~\ref{cor:bss}}\label{sec:bss}
Let $X$ be a topological space such that $H_i(X;\Z)$ is finitely generated for all $i$.
To recover the $2^k$-torsion elements from $H^\ast(X;\Z_2)$, it suffices to
	  understand, which is a classical method, their
	  behaviors under (higher) $\mathrm{mod}$ $2$ Bockstein homomorphisms.
	  More precisely, we have an exact couple
	  \begin{equation}\label{def:bss}
\xymatrix{
    H^\ast(X;\Z) \ar[rr]^{\cdot 2} & & H^\ast(X;\Z) \ar[dl]^{\mod 2}\\
    & H^\ast(X;\Z_2) \ar[ul]^{\kappa} &
}
	  \end{equation}
  where $\kappa$ is the connecting homomorphism, which is
  induced from the exact sequence
$$
\xymatrix{
    0  \ar[rr] && \Z \ar[rr]^{\cdot 2}& &\Z \ar[rr]^{\mod 2}& &\Z_2 \ar[rr]&& 0.
}
$$
	The following fact is well-known (cf.~\cite[Chapter 10]{McCleary2001book}).

\begin{proposition}\label{prop:bss}
    The single-graded spectral sequence $E^\ast_\ast(X)$ associated to \eqref{def:bss}
	satisfies the following properties.
	\begin{enumerate}
	  \item The sequence $E_1^\ast =H^\ast (X;\Z_2)$, and the first differential
	\[d_1\colon H^n(X;\Z_2)\to H^{n+1}(X;\Z_2)\]
	coincides with the Steenrod square $Sq^1$. More explicitly,
	if the $\mathrm{mod}$ $2$ reduction $[c]$ of an integral singular cochain
	$c$ represents a class in $H^n(X;\Z_2)$, then
	\begin{equation}	  \label{eq:d10}
	  d_1[c]=[\frac{1}{2}dc],
	\end{equation}
	where $d$ is the coboundary operator.
  \item The $\Z_{2^{k}}$-summand in $H^{n+1}(X;\Z )$
	induces a $\Z_2$-pair
	in $H^n(X;\Z_2)$ and $H^{n+1}(X;\Z_2)$, respectively,
	which survives to
	$E_k^\ast (X)$, and is
	connected by the $k$th differential $d_k$.
   \item The sequence $E_{\infty}^\ast (X)$
	is the $\mathrm{mod}$ $2$ reduction of the free part of
	$H^\ast (Y;\Z )$.
\end{enumerate}
	\end{proposition}

	Now we apply the Bockstein spectral sequence to a real toric space $Y$.
    First, by Proposition~\ref{prop:BP}, we see that the $\mathrm{mod}$ $2$ cohomology
	groups of $BK$ are generated by monomials of the form
	\[x_{i_1}^{n_1}x_{i_2}^{n_2}\ldots x_{i_\ell}^{n_\ell}, \quad
	  \{i_1,\ldots, i_\ell\}\in K,\]
    on which we have, by the Cartan formula,
	\begin{equation}	  \label{eq:sq}
	  Sq^1(x_{i_1}^{n_1}x_{i_2}^{n_2}\ldots x_{i_\ell}^{n_\ell})=
	  \sum_{j=1}^\ell n_jx_{i_1}^{n_1}\cdots x_{i_{j-1}}^{n_{j-1}}
	  x^{n_j+1}_{i_j}x_{i_{j+1}}^{n_{j+1}}\cdots x_{i_\ell}^{n_\ell}.
	\end{equation}
	  On passage to $Y$ via the inclusion $Y\to BK$, by \eqref{eq:DJ}, an element
	  from $\Z_2[K]$ being a cocycle under $Sq^1$
	  means that its image lies in the
	  ideal $(l_{\lambda_1},\ldots,l_{\lambda_n})$.

	As $\phi$ in Theorem~\ref{thm:main} is a cochain map with respect to $2\ol{d}'$,
	by \eqref{eq:d10} and \eqref{eq:d1}, one can see that
	\begin{equation}
	  \phi(\ol{d}'(r(\sigma_{j}))=\frac{1}{2}\phi(2\ol{d}'(r(\sigma_{j}))
	=Sq^1(x_{r(\sigma_j)}).
	  \label{eq:cal2}
	\end{equation}
    Therefore, together with Proposition \ref{prop:bss} and the second part of Theorem~\ref{thm:integral_cohom_of_Y},
$E^i_{k+1}(Y)$ is isomorphic
		  to $\bigoplus_{\omega\in \row \lambda}E^{i-1}_k(K_{\omega})$,
		  for all $i,k\geq 1$,
which proves Theorem~\ref{cor:bss}.

For the following examples, a simplex $\{i_1,i_2,\ldots,i_\ell\}$ will be denoted by $i_1i_2\ldots i_\ell$ for short.
\begin{example}\label{exm:Kb}
    Let $K$ be the boundary complex of a square having a shelling
    $$
        12,\ 2\ul{3},\ 3\ul{4},\ \ul{1}\ul{4}
    $$
    with their restrictions marked.
    Let $\lambda \colon \Z_2^4 \to \Z_2^2$ be a characteristic function represented by
	$$\Lambda=
		  \begin{pmatrix}
			1 & 0 & 1 & 1\\
			0 & 1 & 0 & 1
		  \end{pmatrix}
    $$
		\[\Lambda=\bordermatrix{%
    &1&	2&	3 & 4 \cr
	&1& 0 & 1 & 1 \cr
    &0& 1 & 0 & 1 \cr
		  },\]
    satisfying the non-singularity condition over $K$, where the numbers above the matrix are indicators for vertices of $K$.
    One can easily see that the real moment-angle complex $\RZ_K$ is the torus $S^1 \times S^1$, and the real toric space $Y=M^\R(K,\lambda)$ associated to $K$ and $\lambda$ is the Klein bottle.
    The cohomology of the Klein bottle $Y$ is known as
    $$
        H^i(Y;\Z) = \left\{
        \begin{array}{ll}
            \Z, & \hbox{for $i=0,1$;} \\
            \Z_2, & \hbox{for $i=2$;} \\
            0& \hbox{otherwise.}
        \end{array}
        \right.
    $$
    Let us compute its Bockstein spectral sequence in two ways.
    First, by Proposition~\ref{prop:DJ}, we have
    $$
        H^\ast (Y;\Z_2)\cong\Z_2[K]/(l_{\lambda_1},l_{\lambda_2})= \Z_2[x_1,\ldots,x_4]/(x_1 x_3,x_2x_4, l_{\lambda_1},l_{\lambda_2}),
    $$
		  where $l_{\lambda_1}=x_1+x_3+x_4$ and $l_{\lambda_2}=x_2+x_4$.
		  By Proposition \ref{prop:basis}, $H^\ast (Y;\Z_2)$ has
		  an additive basis $\{1,x_{3},x_{4},x_1x_4\}$. We see that
		  in $\Z_2[K]$,
		  which is a free $\Z_2[l_{\lambda_1},l_{\lambda_2}]$-module,
		  \begin{align*}
			& Sq^1(x_3)=x_3^2=l_{\lambda_1}x_3+(l_{\lambda_1}+l_{\lambda_2})x_{4}+x_1x_4,
			\numberthis\label{eq:cal1}
			\\
			& Sq^1(x_4)=x_4^2=l_{\lambda_2}x_4, \quad \text{ and }\\
			& Sq^1(x_1x_4)=x_{1}^2x_4+x_1x_4^2=l_{\lambda_1}x_1x_4
		  \end{align*}
		  since $x_1x_3$ and $x_2x_4$ vanish in $\Z_2[K]$.
		  Thus, $x_1x_4$ and $x_3$ are connected by $Sq^1$, whereas
		  $x_4$ and $1$ survive to $E_{\infty}^\ast (Y)$.

    Second, as
$\row \lambda=\{\emptyset, 134, 24, 123\}$ as a set,
under the rule $\omega\cap\sigma=r(\sigma)$ for a facet $\sigma$,
we have $r(12)=\emptyset\in K_{\emptyset}$, $r(34)=4\in K_{24}$ and
$r(23)=3$ and $r(14)=14$ are in $K_{134}$ (with no restrictions in $K_{123}$).
\begin{figure}
   \begin{center}
\begin{tikzpicture}[thick,scale=0.5]
    \fill (0,2) circle (4pt) node[above left] {$1$};
    \node at (4,1) {$\longrightarrow$};
    \newcommand\x{6}
    \fill (\x+0,2) circle (4pt) node[above left] {$1$};
    \fill (\x+2,0) circle (4pt) node[below right] {$3$};
    \draw (\x+2,0) circle (8pt);
    \node at (\x+4,1) {$\longrightarrow$};
    \renewcommand\x{12}
    \fill (\x+0,2) circle (4pt) node[above left] {$1$};
    \fill (\x+2,2) circle (4pt) node[above right] {$4$};
    \fill (\x+2,0) circle (4pt) node[below right] {$3$};
    \draw (\x+2,0) circle (8pt);
    \draw (\x+2,2)-- (\x+2,0);
    \node at (\x+4,1) {$\longrightarrow$};
    \renewcommand\x{18}
    \fill (\x+0,2) circle (4pt) node[above left] {$1$};
    \fill (\x+2,2) circle (4pt) node[above right] {$4$};
    \fill (\x+2,0) circle (4pt) node[below right] {$3$};
    \draw (\x+2,0) circle (8pt);
    \draw (\x+2,2)-- (\x+2,0);
    \draw (\x+1,2) ellipse (1.5 and 0.6);
\end{tikzpicture}
    \end{center}
    \caption{The regular expanding of $K_{134}$ in Example~\ref{exm:Kb}}
    \label{fig:K134}
\end{figure}
\begin{figure}
   \begin{center}
\begin{tikzpicture}[thick,scale=0.5]
    \newcommand\x{0}
    \fill (\x+0,2) circle (4pt) node[above left] {$1$};
    \fill (\x+2,2) circle (4pt) node[above right] {$4$};
    \fill (\x+2,0) circle (4pt) node[below right] {$3$};
    \draw (\x+2,0) circle (8pt);
    \draw (\x+2,2)-- (\x+2,0);
    \draw (\x+1,2) ellipse (1.5 and 0.6);
    \node at (\x+4,1) {$\longrightarrow$};
    \renewcommand\x{6}
    \fill (\x+0,2) circle (4pt) node[above left] {$1$};
    \fill (\x+2,2) circle (4pt) node[above right] {$4$};
    \fill (\x+2,1) circle (4pt) node[below right] {$3$};
    \draw (\x+2,1) circle (8pt);
    \draw (\x+2,2)-- (\x+2,1);
    \draw (\x+1,2) ellipse (1.5 and 0.6);
    \node at (\x+4,1) {$\longrightarrow$};
    \renewcommand\x{12}
    \fill (\x+0,2) circle (4pt) node[above left] {$1$};
    \fill (\x+2,2) circle (4pt) node[above right] {$3(4)$};
    \draw (\x+2,2) circle (8pt);
    \draw (\x+1,2) ellipse (1.5 and 0.6);
\end{tikzpicture}
    \end{center}
    \caption{The retraction of $K_{134}$ in Example~\ref{exm:Kb}}
    \label{fig:K134_retraction}
\end{figure}

The shelling above gives a regular expanding sequence of $K_{134}$,
i.e., $1$, $3$, $34$, $14$ by Proposition~\ref{prop:esKw}, in which $3$ and $14$
are critical (see Figure~\ref{fig:K134}).
The simplicial retraction maps
$34$ to
$3$ as in Figure~\ref{fig:K134_retraction}, therefore
  $\rho([34])=0$ and $\rho([4])=[3]$ (see Lemmas~\ref{lem:rho1} and~\ref{prop:key2}).
After dualization, we have
$$
    \ol{d}'3^\ast =d'4^\ast =(14)^\ast,
$$
which coincides with the previous calculation \eqref{eq:cal1} by
\eqref{eq:cal2}.
\end{example}

\begin{example}\label{exm:K3}
    Let $K$ be the boundary of the simplicial $3$-polytope shown in Figure~\ref{fig:K3} and let $\lambda \colon \Z_2^7 \to \Z_2^3$ be a characteristic function represented by
		\[\Lambda=\bordermatrix{%
    &1&	2&	3 & 4 & 5 & 6 &7\cr
	&1& 1 & 0 & 0 & 0 & 0 &1\cr
    &0& 0 & 1 & 1 & 1 & 0 &0\cr
	&0& 0 & 0 & 0 & 1 & 1 &1\cr
		  },\]
    where the numbers above the matrix are indicators for vertices of $K$.
\begin{figure}
   \begin{center}
\begin{tikzpicture}[thick]
    \coordinate (v1) at (3,3.3) {};
    \coordinate (v2) at (4,0) {};
    \coordinate (v3) at (1.7, 0.5) {};
    \coordinate (v4) at (5,2.5) {};
    \coordinate (v5) at (5,1) {};
    \coordinate (v6) at (0.9,2) {};
    \coordinate (v7) at (3.5,1.8) {};

    \draw (v1) node[above] {$1$};
    \draw (v2) node[below right] {$2$};
    \draw (v3) node[below left] {$3$};
    \draw (v4) node[above right] {$4$};
    \draw (v5) node[right] {$5$};
    \draw (v6) node[left] {$6$};
    \draw (v7) node[left] {$7$};

    \draw (v1)--(v6)--(v3)--(v2)--(v5)--(v4)--(v1);
    \draw (v1)--(v3)--(v5)--(v1)--(v7)--(v3)--(v7)--(v5);
    \draw[style=dashed] (v4)--(v6)--(v2)--(v4);
\end{tikzpicture}
   \end{center}
   \caption{}
   \label{fig:K3}
\end{figure}
	We give $K$ a shelling with restrictions
	marked as
	\begin{align*}
	  137, \ 13\underline{6},\
	  3\underline{5}7,\ \underline{1}\underline{5}7,\
	  \underline{2}35, \
	  \underline{2}3\underline{6}, \
	  1\underline{4}5,\ \underline{2}\underline{4}5, \
	  1\underline{4}\underline{6}, \text{ and }
	  \underline{2}\underline{4}\underline{6}.
	\end{align*}
	With the $10$ basis elements $\{1, x_2, x_4, x_5, x_6, x_1 x_5, x_2 x_6, x_2 x_4, x_4 x_6, x_2 x_4 x_6\}$ of $H^\ast(Y;\Z_2)$ associated
	to the restrictions above, one can calculate their image under $Sq^1$
	in $\Z_2[K]$;
    we have
    $x_2^2=l_{\lambda_1}x_2$,
	$x_4^2=(l_{\lambda_2}+l_{\lambda_3})x_4+x_4 x_6$,
    $x_5^2=(l_{\lambda_1}+l_{\lambda_3})x_5+l_{\lambda_3}x_2+x_1 x_5+x_2 x_6$,
    $x_6^2=l_{\lambda_3}x_{6}$ together with
    $Sq^1(x_1 x_5)=(l_{\lambda_1}+l_{\lambda_3})x_1 x_5$,
	$Sq^1(x_2 x_6)=(l_{\lambda_1}+l_{\lambda_3})x_2 x_6$,
    $Sq^1(x_2 x_4)=(l_{\lambda_1}+l_{\lambda_2}+l_{\lambda_3})x_2 x_4+x_2 x_4 x_6$,
    $Sq^1(x_4 x_6)=(l_{\lambda_2}+l_{\lambda_3})x_4 x_6$,
    and
    $Sq^1(x_2 x_4 x_6) = (l_{\lambda_1}+l_{\lambda_2}+l_{\lambda_3})x_2 x_4 x_6$
    where
	  $l_{\lambda_1}=x_1+x_2+x_7$, $l_{\lambda_2}=x_3+x_4+x_5$ and
	  $l_{\lambda_3}=x_5+x_6+x_7$.

	  It is interesting to compare the (simplicial) coboundary $\ol{d}'$
	  between critical simplices in each full subcomplexes $K_{\omega}$ for $\omega\in \row \lambda$,
	  and the Bockstein homomorphism $Sq^1$
	  by \eqref{eq:cal2}.
We illustrate this in Figure \ref{fig:E6},
	  where $\emptyset\in K_{\emptyset}$ is ignored; notice that
	  no restrictions are contained in full subcomplexes $K_{345}$
	  and $K_{123457}$.
\begin{figure}
   \begin{center}
\begin{tikzpicture}[thick,scale=0.5]
    \newcommand\di{4}
    \newcommand\x{-0.5}
    \newcommand\y{9}
    \draw (\x-1,\y-\di-0.7) rectangle (\x+5*\di+2,\y+0.7);
    \node (1) at (\x+0,\y+0) {$1$};
    \node (x2) at (\x+1*\di,\y+0) {$x_2$};
    \node (x6) at (\x+0.5*\di,\y-\di) {$x_6$};
    \node (x26) at (\x+2*\di,\y-\di) {$x_2 x_6$};
    \node (x5) at (\x+2.5*\di,\y+0) {$x_5$};
    \node (x15) at (\x+3*\di,\y-\di) {$x_1 x_5$};
    \node (x24) at (\x+4*\di,\y+0) {$x_2 x_4$};
    \node (x246) at (\x+4*\di,\y-\di) {$x_2 x_4 x_6$};
    \node (x4) at (\x+5*\di,\y+0) {$x_4$};
    \node (x46) at (\x+5*\di,\y-\di) {$x_4 x_6$};
    \path (x5) edge node[left]{$Sq^1$} (x26);
    \path (x5) edge node[right]{$Sq^1$} (x15);
    \path (x24) edge node[right]{$Sq^1$} (x246);
    \path (x4) edge node[right]{$Sq^1$} (x46);

    \renewcommand\x{0}
    \renewcommand\y{0}
    \fill (\x+0,\y+0) circle (4pt) node[left] {$2$};
    \draw (\x+0,\y+0) circle (8pt);
    \fill (\x+2,\y+0) circle (4pt) node[right] {$7$};
    \fill (\x+1,\y+1.5) circle (4pt) node[above] {$1$};
    \draw (\x+2,\y+0) --(\x+1,\y+1.5);

    \renewcommand\x{6}
    \renewcommand\y{0}
    \fill (\x+0,\y+0) circle (4pt) node[left] {$6$};
    \draw (\x+0,\y+0) circle (8pt);
    \fill (\x+2,\y+0) circle (4pt) node[right] {$7$};
    \fill (\x+1,\y+1.5) circle (4pt) node[above] {$5$};
    \draw (\x+2,\y+0) --(\x+1,\y+1.5);

    \renewcommand\x{12}
    \renewcommand\y{0}
    \fill (\x+0,\y+0) circle (4pt) node[left] {$3$};
    \fill (\x+2,\y+0) circle (4pt) node[right] {$4$};
    \fill (\x+1,\y+1.5) circle (4pt) node[above] {$5$};
    \draw (\x+0,\y+0)--(\x+1,\y+1.5)--(\x+2,\y+0);

    \renewcommand\x{18}
    \renewcommand\y{0}
    \fill (\x+0,\y+0) circle (4pt) node[below left] {$2$};
    \fill (\x+2,\y+0) circle (4pt) node[below right] {$5$};
    \draw (\x+2,\y+0) circle (8pt);
    \fill (\x+2,\y+2) circle (4pt) node[above right] {$1$};
    \fill (\x+0,\y+2) circle (4pt) node[above left] {$6$};
    \draw (\x+0,\y+0)--(\x+2,\y+0)--(\x+2,\y+2) -- (\x+0,\y+2) -- cycle;
    \draw (\x+2,\y+1) ellipse (0.6 and 1.4);
    \draw (\x+0,\y+1) ellipse (0.6 and 1.4);

    \renewcommand\x{9}
    \renewcommand\y{-5.5}
    \fill (\x+0,\y+0) circle (4pt) node[below left] {$3$};
    \fill (\x+2,\y+0) circle (4pt) node[below right] {$7$};
    \fill (\x+2,\y+2) circle (4pt) node[above right] {$4$};
    \draw (\x+2,\y+2) circle (8pt);
    \fill (\x+0,\y+2) circle (4pt) node[above left] {$6$};
    \draw (\x+2,\y+2) -- (\x+0,\y+2) -- (\x+0,\y+0)--(\x+2,\y+0);
    \draw (\x+1,\y+2) ellipse (1.4 and 0.6);

    \renewcommand\x{1}
    \renewcommand\y{-6}
    \draw [fill=black!20!white] (\x+0,\y+0) rectangle (\x+3, \y+3);
    \fill (\x+0,\y+0) circle (4pt) node[below left] {$2$};
    \fill (\x+3,\y+0) circle (4pt) node[below right] {$3$};
    \fill (\x+3,\y+3) circle (4pt) node[above right] {$1$};
    \fill (\x+0,\y+3) circle (4pt) node[above left] {$4$};
    \fill (\x+1.5,\y+1.5) circle (4pt) node[below] {$6$};
    \draw (\x+3,\y+3) -- (\x+0,\y+3) -- (\x+0,\y+0)--(\x+3,\y+0)--(\x+3,\y+3)--(\x+1.5,\y+1.5)--(\x+0,\y+3)--(\x+1.5,\y+1.5) -- (\x+0,\y+0)--(\x+1.5,\y+1.5)--(\x+3,\y+0);
    \draw (\x+0,\y+1.5) ellipse (0.4 and 1.8);
    \draw (\x+0.5,\y+1.5) ellipse (1.4 and 2.1);

    \renewcommand\x{15.5}
    \renewcommand\y{-6}
    \draw [fill=black!20!white] (\x+0,\y+0) rectangle (\x+4, \y+3);
    \fill (\x+0,\y+0) circle (4pt) node[below left] {$3$};
    \fill (\x+4,\y+0) circle (4pt) node[below right] {$2$};
    \fill (\x+4,\y+3) circle (4pt) node[above right] {$4$};
    \fill (\x+0,\y+3) circle (4pt) node[above left] {$1$};
    \fill (\x+1.3,\y+1.5) circle (4pt) node[left] {$7$};
    \fill (\x+2.7,\y+1.5) circle (4pt) node[right] {$5$};
    \draw (\x+0,\y+0) -- (\x+1.3,\y+1.5) -- (\x+0,\y+3) -- (\x+2.7,\y+1.5) -- (\x+0, \y+0)
          (\x+1.3,\y+1.5) -- (\x+2.7,\y+1.5)-- (\x+4,\y+3)
          (\x+2.7,\y+1.5)-- (\x+4,\y+0);
\end{tikzpicture}
   \end{center}
   \caption{The relation between simplicial and cellular cochains}
    \label{fig:E6}
\end{figure}

	  Since $\wt{H}^\ast (K_{\omega})$
	  is torsion-free for all
	  $\omega\in \row \lambda$,
by Theorems~\ref{thm:main} and ~\ref{cor:bss}, we can choose $1$, $x_2$, $x_6$ and $x_2 x_6$ as
	  a basis for $E_2^\ast(Y)=E_{\infty}^\ast(Y)$, and the other four basis elements in $H^\ast (Y;\Z_2)$
	  are connected in pairs by $Sq^1$.
    As a conclusion,
	  \[ H^k(Y;\Z)\cong\begin{cases}
		  \Z    & k=0\\
		  \Z\oplus\Z & k=1\\
		  \Z\oplus\Z_2\oplus\Z_2 & k=2\\
          \Z_2 & k=3.
		\end{cases}\]
	\end{example}

  \begin{remark}
  When $K$ is shellable in the sense of
  Bj\"{o}rner and Wachs \cite{Bjorner-Waches1996} (i.e., $K$ need not to be pure),
  the algebraic basis from shellability (cf. Proposition~\ref{prop:basis}) still holds
  (see \cite[Theorem 12.3]{Bjorner-Waches1997}),
  and our proof of Theorem~\ref{thm:main}
  works as well. Shellable complexes (in the usual sense) are considered
  in this paper since we are interested in the case when the small
  cover is a piecewise linear manifold (see Proposition~\ref{prop:dec}).
\end{remark}

\section*{Acknowledgments}
  The authors would like to thank Professor Anthony Bahri for the invitation of the visit to Rider University, and for many inspiring conversations.
  The second named author also thanks Professor Alexandru~I. Suciu at Northeastern University for providing a wonderful sabbatical location conducive to thinking deep thoughts.

\bigskip

\bibliographystyle{amsplain}

\begin{thebibliography}{10}

\bibitem{Bjorner-Waches1996}
Anders Bj{\"o}rner and Michelle~L. Wachs, \emph{Shellable nonpure complexes and
  posets. {I}}, Trans. Amer. Math. Soc. \textbf{348} (1996), no.~4, 1299--1327.
  \MR{1333388 (96i:06008)}

\bibitem{Bjorner-Waches1997}
Anders Bj\"{o}rner and Michelle~L. Wachs, \emph{Shellable nonpure complexes and
  posets. {II}}, Trans. Amer. Math. Soc. \textbf{349} (1997), no.~10,
  3945--3975. \MR{1401765}

\bibitem{Bruggesser-Mani1971}
H.~Bruggesser and P.~Mani, \emph{Shellable decompositions of cells and
  spheres}, Math. Scand. \textbf{29} (1971), 197--205 (1972). \MR{0328944}

\bibitem{BEMPP}
Victor Buchstaber, Nikolay Erokhovets, Mikiya Masuda, Taras Panov, and
  Seonjeong Park, \emph{Cohomological rigidity of manifolds defined by
  right-angled $3$-dimensional polytopes}, preprint, arXiv:1610.07575, 2016.

\bibitem{Cai2017}
Li~Cai, \emph{On products in a real moment-angle manifold}, J. Math. Soc. Japan
  \textbf{69} (2017), no.~2, 503--528. \MR{3638276}

\bibitem{Choi-Kaji-Theriault2017}
Suyoung Choi, Shizuo Kaji, and Stephen Theriault, \emph{Homotopy decomposition
  of a suspended real toric space}, Bol. Soc. Mat. Mex. (3) \textbf{23} (2017),
  no.~1, 153--161. \MR{3633130}

\bibitem{Choi-Park2017_multiplicative}
Suyoung Choi and Hanchul Park, \emph{Multiplicative structure of the cohomology
  ring of real toric spaces}, preprint, arXiv:1711.04983 (2017).

\bibitem{Choi-Park2017_torsion}
\bysame, \emph{On the cohomology and their torsion of real toric objects},
  Forum Math. \textbf{29} (2017), no.~3, 543--553. \MR{3641664}

\bibitem{Danilov1978}
V.~I. Danilov, \emph{The geometry of toric varieties}, Uspekhi Mat. Nauk
  \textbf{33} (1978), no.~2(200), 85--134, 247. \MR{495499 (80g:14001)}

\bibitem{Davis-Januszkiewicz1991}
Michael~W. Davis and Tadeusz Januszkiewicz, \emph{Convex polytopes, {C}oxeter
  orbifolds and torus actions}, Duke Math. J. \textbf{62} (1991), no.~2,
  417--451. \MR{1104531 (92i:52012)}

\bibitem{Eilenberg-Zilber1953}
Samuel Eilenberg and J.~A. Zilber, \emph{On products of complexes}, Amer. J.
  Math. \textbf{75} (1953), 200--204. \MR{0052767}

\bibitem{Etingof-Henroques-Kamnitzer-Rains2010}
Pavel Etingof, Andr\'e Henriques, Joel Kamnitzer, and Eric~M. Rains, \emph{The
  cohomology ring of the real locus of the moduli space of stable curves of
  genus 0 with marked points}, Ann. of Math. (2) \textbf{171} (2010), no.~2,
  731--777. \MR{2630055}

\bibitem{Ishida-Fukukawa-Masuda2013}
Hiroaki Ishida, Yukiko Fukukawa, and Mikiya Masuda, \emph{Topological toric
  manifolds}, Mosc. Math. J. \textbf{13} (2013), no.~1, 57--98, 189--190.
  \MR{3112216}

\bibitem{Jurkiewicz1980}
J.~Jurkiewicz, \emph{Chow ring of projective nonsingular torus embedding},
  Colloq. Math. \textbf{43} (1980), no.~2, 261--270 (1981). \MR{628181
  (82m:14025)}

\bibitem{Jurkiewicz1985}
Jerzy Jurkiewicz, \emph{Torus embeddings, polyhedra, {$k^\ast$}-actions and
  homology}, Dissertationes Math. (Rozprawy Mat.) \textbf{236} (1985), 64.
  \MR{820078}

\bibitem{McCleary2001book}
John McCleary, \emph{A user's guide to spectral sequences}, second ed.,
  Cambridge Studies in Advanced Mathematics, vol.~58, Cambridge University
  Press, Cambridge, 2001. \MR{1793722}

\bibitem{Munkres1984book}
James~R. Munkres, \emph{Elements of algebraic topology}, Addison-Wesley
  Publishing Company, Menlo Park, CA, 1984. \MR{755006}

\bibitem{Reisner1976}
Gerald~Allen Reisner, \emph{Cohen-{M}acaulay quotients of polynomial rings},
  Advances in Math. \textbf{21} (1976), no.~1, 30--49. \MR{0407036}

\bibitem{Rourke-Sanderson1972book}
C.~P. Rourke and B.~J. Sanderson, \emph{Introduction to piecewise-linear
  topology}, Springer-Verlag, New York-Heidelberg, 1972, Ergebnisse der
  Mathematik und ihrer Grenzgebiete, Band 69. \MR{0350744}

\bibitem{Stanley1996book}
Richard~P. Stanley, \emph{Combinatorics and commutative algebra}, second ed.,
  Progress in Mathematics, vol.~41, Birkh\"auser Boston, Inc., Boston, MA,
  1996. \MR{1453579}

\bibitem{ST2012}
Alexander~I. Suciu and Alvise Trevisan, \emph{Real toric varieties and abelian
  covers of generalized davis-januszkiewicz spaces}, 2012.

\bibitem{Suyama2014}
Yusuke Suyama, \emph{Examples of smooth compact toric varieties that are not
  quasitoric manifolds}, Algebr. Geom. Topol. \textbf{14} (2014), no.~5,
  3097--3106. \MR{3276857}

\bibitem{Trevisan2012}
Alvise Trevisan, \emph{Generalized Davis–Januszkiewicz spaces and their applications in algebra
and topology}, Ph.D. thesis, Vrije University Amsterdam, 2012; available at
\href{http://dspace.ubvu.vu.nl/handle/1871/32835}{http://dspace.ubvu.vu.nl/handle/1871/32835}.

\bibitem{Vesnin1987}
A.~Yu. Vesnin, \emph{Three-dimensional hyperbolic manifolds of {L}\"obell
  type}, Sibirsk. Mat. Zh. \textbf{28} (1987), no.~5, 50--53. \MR{924975}

\bibitem{Ziegler1995book}
G{\"u}nter~M. Ziegler, \emph{Lectures on polytopes}, Graduate Texts in
  Mathematics, vol. 152, Springer-Verlag, New York, 1995. \MR{1311028}

\end{thebibliography}
\providecommand{\bysame}{\leavevmode\hbox to3em{\hrulefill}\thinspace}
\providecommand{\MR}{\relax\ifhmode\unskip\space\fi MR }
\providecommand{\MRhref}[2]{%
  \href{http://www.ams.org/mathscinet-getitem?mr=#1}{#2}
}
\providecommand{\href}[2]{#2}

\end{document}